\documentclass[11pt]{article}
\usepackage{graphicx}
\usepackage{graphics}
\usepackage{amsfonts}
\usepackage{amscd}
\usepackage{amsthm}
\usepackage{indentfirst}
\usepackage[all]{xy}
\usepackage{amssymb, amsmath, amsthm, amsgen, amstext, amsbsy, amsopn}
\usepackage{epic,eepic}
\usepackage{hyperref}
\usepackage{enumitem} 
\usepackage{color}
\usepackage{ dsfont }

\theoremstyle{plain}
\newtheorem{thm}{Theorem}[section]
\newtheorem{lemma}[thm]{Lemma}
\newtheorem{prop}[thm]{Proposition}

\newtheorem{cor}[thm]{Corollary}
\newtheorem*{thmsymp*}{Theorem \ref{thm:gen_cover}'}

\theoremstyle{definition}
\newtheorem{defin}[thm]{Definition}
\newtheorem{rem}[thm]{Remark}
\newtheorem{example}[thm]{Example}
\newtheorem{conj}[thm]{Conjecture}

\newcommand{\R}{{\mathbb{R}}}

\newcommand{\Z}{{\mathbb{Z}}}
\newcommand{\N}{{\mathbb{N}}}

\newcommand{\cD}{{\mathcal{D}}}
\newcommand{\cC}{{\mathcal{C}}}

\newcommand{\cF}{{\mathcal{F}}}
\newcommand{\cG}{{\mathcal{G}}}

\newcommand{\cI}{{\mathcal{I}}}
\newcommand{\cJ}{{\mathcal{J}}}

\newcommand{\cU}{{\mathcal{U}}}
\newcommand{\cV}{{\mathcal{V}}}

\def\id{{1\hskip-2.5pt{\rm l}}}

\newcommand{\supp}{{\it supp\,}}

\definecolor{lev}{rgb}{0.773,0.294,0.549}

\begin{document}

\title{Poisson Brackets of Partitions of Unity on Surfaces}

\author{Lev Buhovsky, Alexander Logunov, Shira Tanny}

\maketitle


\begin{abstract}
        Given an open cover of a closed symplectic manifold, 
        consider all smooth partitions of unity consisting of functions supported in the covering sets.  The Poisson 
        bracket invariant of the cover measures how much the functions from such a partition of unity can become 
        close to being Poisson commuting. We introduce a new approach to this invariant, which enables us to 
        prove the lower bound conjectured by L. Polterovich, in dimension $2$. 
\end{abstract}

\section{Introduction and results.}

Let $(M,\omega)$ be a closed connected symplectic manifold and let $\mathcal U:= \{U_i\}_{i\in I}$ be a finite open cover of $M$ by displaceable\footnote{We say that a subset $S\subset M$ is {\it displaceable} if there exists a Hamiltonian diffeomorphism $\phi:M\rightarrow M$ that displaces its closure, namely $\phi(\bar S)\cap \bar S=\emptyset$.} sets. Any subordinate\footnote{Given an open cover $\mathcal U:= \{U_i\}_{i\in I}$  of $M$, we say that a partition of unity  $\mathcal F =\{f_i\}_{i\in I}$ is subordinate to $\mathcal U$ if $supp(f_i)\subset U_i$ for all $i\in I$.} partition of unity $\mathcal F =\{f_i\}_{i\in I}$ cannot be Poisson commuting, as follows from the nondisplaceable fiber theorem \cite{entov2004quasi}.
Note that the assumption on the displaceability of sets in $\mathcal U$ is crucial - any partition of unity on $S^2\subset \R^3$ that depends only on the height $z$ is Poisson commuting. The study of lower bounds for this non-commutativity was initiated in \cite{entov2006quasi}, where M. Entov, L. Polterovich and F. Zapolsky used symplectic quasi-states to prove that $\max_{i,j} \|\{f_i,f_j\}\|\geq \text{const}/|I|^3$. Here and further on, $\|\cdot\|:C^\infty(M,\R)\rightarrow\R$ stands for the uniform (or the $ L^\infty $) norm, $\|f\|=\max_M |f|$. Below, we present an improvement of this bound for the case where $M$ is a surface, see Corollary~\ref{cor:c0_bound}.

The non-commutativity of partitions of unity subordinate to a cover $\mathcal{U}$ can be also measured by the Poisson bracket invariant, which was introduced by L. Polterovich in \cite{polterovich2012quantum}:
\begin{equation}\label{eq:pb_def}
pb(\mathcal{U}):= \inf_{\mathcal F} \max_{x,y\in[-1,1]^{|I|}} \Big\|\Big\{\sum_{i\in I} x_i f_i,\sum_{j \in I} y_j f_j\Big\}\Big\|,
\end{equation}
where the infimum is taken over all partitions on unity $\mathcal F$ subordinate to $\mathcal U$. In \cite{polterovich2012quantum, polterovich2014symplectic}, Polterovich	explained the relations between this invariant and quantum mechanics and conjectured an optimal lower bound for $pb(\cU)$ in terms of the {\it magnitude of localization} of $\cU$: 

\begin{conj}\label{con:pb}
	Let $(M,\omega)$ be a closed symplectic manifold, and let $\cU = \{U_i\}_{i\in I}$ be an open cover of $M$ by displaceable sets. Then, there exists a constant $C = C(M,\omega) > 0 $ depending only  on the symplectic manifold, such that 
	\begin{equation}
	pb(\cU) \geq \frac{C}{e(\cU)},
	\end{equation}  
	where $e(\cU): = \max_{i\in I} e(U_i)$ and $e(U_i)$ is the displacement energy\footnote{For a displaceable subset $S\subset M$, the {\it displacement energy} of $ S $ 
	is the infimum of a Hofer length $ \ell_{\text{Hof}}(H) = \int_{0}^1 \max_M H(\cdot,t) - \min_M H(\cdot,t) \, dt $, for all time-dependent smooth Hamiltonian functions 
	$ H : M \times [0,1] \rightarrow \mathbb{R} $ such that the time-$1$ map $ \phi : M \rightarrow M $ of the Hamiltonian flow generated by $ H $, displaces 
	the closure of $ S $:  $\phi(\bar S)\cap \bar S=\emptyset$.} of $U_i$.	
\end{conj}
 Polterovich also proved several lower bounds for this invariant, which were then improved and extended by S. Seyfaddini in \cite{seyfaddini2014spectral} as well as by S. Ishikawa in \cite{ishikawa2015spectral}. These lower bounds decay in the {\it degree} of the cover (which was defined in  \cite{polterovich2014symplectic}), and their proofs rely on ``hard" symplectic topology (for example, properties of spectral invariants).  In this paper, we prove Conjecture~\ref{con:pb} in dimension 2 using only elementary arguments.

\begin{rem}\label{rem:energy_vs_area}
	For a closed symplectic surface $(M,\omega)$, a connected subset $S\subset M$ is displaceable if and only if it is contained in an embedded open topological disc $ V \subset M $ with smooth boundary and $ area(V) < \frac{area(M)}{2} $. In this case, the infimum of the area of such a topological disc $ V $ is precisely the displacement energy $ e(S) $. If a subset $ S \subset M $ is not displaceable then we have $ e(S) = +\infty $.
\end{rem}

The following lemma holds for manifolds of general dimension, but we will apply it to closed surfaces.
\begin{lemma}\label{lem:pb_vs_sum}
	Let $(M^{2n},\omega)$ be a closed symplectic manifold of dimension $2n$. Then, there exists a constant $c(n) > 0$ depending only on the dimension, such that for every finite collection of smooth functions $\{f_i\}_{i\in I}$ on $M$,
	\begin{equation*}\label{eq:pb_vs_sum}
	\max_{x,y\in[-1,1]^{|I|}} \Big\|\Big\{\sum_{i\in I} x_i f_i,\sum_{j\in I} y_j f_j\Big\}\Big\|
	\geq c(n)\cdot \max_M \sum_{i,j\in I}|\{f_i,f_j\}|.
	\end{equation*} 
\end{lemma}
In fact, we prove that a pointwise inequality holds, see Appendix~\ref{app:lin_exc}. 
In Section \ref{S:Main-results} we prove lower bounds for the $L^\infty$ and the $L^1$ norms of the sum $\sum_{i,j\in I}|\{f_i,f_j\}|$ on a closed symplectic surface $(M^2,\omega)$, and use Lemma~\ref{lem:pb_vs_sum} for the case where $n=1$ to conclude that the same holds for $pb(\mathcal U)$ up to a constant. \\

\subsection{Poisson bracket on surfaces.} \label{S:Main-results}

The present subsection contains the main results of this paper (Theorems \ref{thm:gen_cover} and \ref{thm:essential}) concerning symplectic geometry in dimension two, and Sections \ref{sec:essential_sets} and \ref{sec:gen_covers} are devoted for their proofs. The formulations and proofs of the main results  do not assume any knowledge in symplectic geometry, beyond explained in Remark \ref{rem:basic-symp-notions} below.

\begin{rem} \label{rem:basic-symp-notions}
Given a surface $ M $, endowed with an area form $ \omega $ (in that case we say that $ (M,\omega) $ is a symplectic surface), the Poisson bracket of a pair of smooth functions on $ M $, is itself a smooth function on $ M $, which measures how much the differentials of the functions are non-collinear at each point. More precisely, given $ f,g \in C^\infty(M) $, their Poisson bracket $ \{f,g\} \in C^\infty(M) $ is defined by $ df \wedge dg = \{ f, g \} \, \omega $. For example, if $ M = \mathbb{R}^2 $ with coordinates $ (x,y) $, and $ \omega = dx \wedge dy $ is the standard area form, then $ \{ f,g \} = f_x g_y - f_y g_x $ is the determinant of the $2 \times 2 $ matrix whose rows are the gradients of $ f $ and $ g $. In higher dimensions, the Poisson bracket is naturally defined on any symplectic manifold, and we refer the interested reader to \cite{arnold1989classical-mechanics} for details.
\end{rem}

Let $ (M,\omega) $ be a closed connected symplectic surface. Recall that given an open cover $\mathcal U:= \{U_i\}_{i\in I}$ of $M$, we say that a partition of unity  $\mathcal F =\{f_i\}_{i\in I}$ is subordinate to $\mathcal U$ if $supp(f_i)\subset U_i$ for all $i\in I$. 
As before, we denote by $\|\cdot\|:C^\infty(M,\R)\rightarrow\R$ the uniform norm, $\|f\|=\max_M |f|$. Let us pass to our first main result.

\begin{thm}\label{thm:gen_cover}
	Let $ (M,\omega) $ be a closed and connected symplectic surface. Let $\{f_i\}_{i\in I}$,  $\{g_j\}_{j\in J}$ be partitions 
	of unity on $ M $, such that for some real number $ 0 < A < area(M)/2 $, the support of 
	each $ f_i $ lies in some topological disc of area not greater than $ A $, and similarly, the support of each $ g_j $ lies in some topological disc of area not greater than $ A $. Then,
	\begin{equation}\label{eq:gen_cover}
	\sum_{i\in I}\sum_{j\in J} \int_M |df_i \wedge dg_j| = \sum_{i\in I}\sum_{j\in J} \int_M |\{f_i,g_j\}|\ \omega \geq \frac{area(M)}{2A}
	\end{equation}
\end{thm}

\noindent Our second main result is applicable only to a certain class of covers.

\begin{defin}
	Given an open cover  $\mathcal{U} = \{U_i\}_{i\in I}$ of $M$, we say that a set  $U_\ell\in\mathcal U$ is {\it essential} if $\mathcal U\setminus\{U_\ell\}$ is not a cover, that is, $\cup_{i\neq \ell} U_i\neq M$.
	We denote by $I_{ess}(\mathcal U)\subset I$ the subset of indices corresponding to essential sets in $\mathcal U$.
\end{defin}

\begin{thm}\label{thm:essential}
         Let $ (M,\omega) $ be a closed and connected symplectic surface. Let $\mathcal U:= \{U_i\}_{i\in I}$ be an open cover of $M$ by topological discs of area 
	less than $ area(M)/2 $, and let $\mathcal F =\{f_i\}_{i\in I}$ be any partition of unity subordinate to $ \cU $. Then,
		\begin{eqnarray}
		\int_M \sum_{i,j\in I}  |\{f_i,f_j\}|\ \omega&\geq& |I_{ess}(\mathcal U)|, \label{eq:ess_bnd1}\\
		 \max_M\sum_{i,j\in I} |\{f_i, f_j\}| &\geq& \frac{1}{\min_{\ell\in I_{ess}(\mathcal U)} area(U_\ell)}, \label{eq:ess_bnd2}
		 \end{eqnarray}
		 where we set the minimum of an empty set to be infinity.

\end{thm}
\begin{rem}\label{rem:essential}\
	\begin{itemize}
		\item Applying Lemma~\ref{lem:pb_vs_sum} to the lower bounds (\ref{eq:ess_bnd1}), (\ref{eq:ess_bnd2}), we get corresponding lower 
		bounds for the Poisson bracket invariant $pb(\cU)\,$:
		\begin{eqnarray}
		 pb(\cU) &\geq& \frac{c \cdot |I_{ess}(\mathcal U)|}{area(M)}, \label{eq:ess_bnd-pb1}\\
		 pb(\cU) &\geq& \frac{c}{\min_{\ell\in I_{ess}(\mathcal U)} area(U_\ell)}, \label{eq:ess_bnd-pb2}
		 \end{eqnarray} 
		 for an absolute constant $ c>0 $.
		\item  If $\mathcal U$ is a minimal cover,  every set is essential and thus $I_{ess}(\mathcal U) =I$. In this case Theorem~\ref{thm:essential} implies 
		that $\int_M \sum_{i,j\in I} |\{f_i,f_j\}|\ \omega\geq {|I|}$, and $\max_M\sum_{i,j\in I} |\{f_i, f_j\}| \geq 1/(\min_{i\in I} area(U_i))$.  
		\item When the cover $\mathcal U$ has no essential sets, $I_{ess}(\mathcal U)=\emptyset$ and Theorem~\ref{thm:essential} gives a trivial lower bound for sum of Poisson brackets. \\
	\end{itemize}
\end{rem}

Theorem \ref{thm:gen_cover} can be reformulated in terms of a cover (whereas now, the cover can be general, i.e. it does not require to admit essential sets or 
to consist only of topological discs):

\begin{thmsymp*} 
	Let $ (M,\omega) $ be a closed and connected symplectic surface. Let $\mathcal U=\{U_i\}_{i\in I}$, $\mathcal V=\{V_j\}_{j\in J}$ be {finite} 
	open covers of $M$, and let $\{f_i\}_{i\in I}$,  $\{g_j\}_{j\in J}$ be partitions of unity subordinate to $\mathcal U$, $\mathcal V$ correspondingly. Then,
	\begin{equation}\label{eq:gen_cover-reformulation}
	\int_M \sum_{i\in I}\sum_{j\in J} |\{f_i,g_j\}|\ \omega \geq \frac{area(M)}{2\cdot \max(e(\mathcal U), e(\mathcal V))}
	\end{equation}
\end{thmsymp*}

Here $e(\cU) = \max_{i\in I} e(U_i)$ and $e(U_i)$ is the displacement energy\footnote{See Remark \ref{rem:energy_vs_area} regarding the notion of the displacement energy in dimension $ 2 $.} of $U_i$ (resp., $e(\cV) = \max_{j\in I} e(V_j)$ and $e(V_j)$ is the displacement energy of $V_j$). See Remark \ref{rem:equivalence} for an explanation of equivalence of Theorems \ref{thm:gen_cover} and \ref{thm:gen_cover}'.

Applying the theorem for $\cU=\cV$ and $\{f_i\} = \{g_j\}$, and using Lemma~\ref{lem:pb_vs_sum}, we obtain the affirmative answer to Conjecture~\ref{con:pb} in dimension $ 2 $, 
as a corollary:
\begin{cor}\label{cor:pb_bound_gen}
	Let $ (M,\omega) $ be a closed and connected symplectic surface. Let $\cU=\{U_i\}_{i\in I}$ be an open 
	displaceable cover of $M$, then for an absolute constant $c>0$ we have
	\begin{equation*}\label{eq:pb_bound_gen}
	pb(\cU)\geq \frac{c}{e(\mathcal U)}.
	\end{equation*}
\end{cor}

\begin{rem}
The bound in Corollary \ref{cor:pb_bound_gen}, and bounds (\ref{eq:ess_bnd-pb1}) and (\ref{eq:ess_bnd-pb2}) in Remark \ref{rem:essential}, are sharp in the following sense: on every closed symplectic surface $(M,\omega)$, one can construct a sequence of open displaceable covers $\{\mathcal U_k\}_{k\in\N}$, such that $|I_{ess}(\mathcal U_k)| \approx k$, $\min_{i\in I} area(U_i)\approx \max_{i\in I} area(U_i)\approx1/k$ and  $pb(\mathcal U_k)\approx k$. See {Example~\ref{exa:essential_sharp}} for details.
\end{rem}

The following definition of a degree of a cover is slightly different than the one present by Polterovich in \cite{polterovich2014symplectic}. In fact, the degree below is not larger, and therefore lower bounds with respect to it hold also for the standard definition.
\begin{defin}
	Given a cover $\cU=\{U_i\}_{i\in I}$ of $M$, we define its {\it degree} to be
	\begin{equation*}
	d:=\max_{x\in M} \#\{i\in I:\ x\in U_i\}.
	\end{equation*}
\end{defin}

\begin{cor} \label{cor:c0_bound}
	Let $ (M,\omega) $ be a closed and connected symplectic surface. Let $\cU=\{U_i\}_{i\in I}$ be open displaceable 
	cover of $M$ and let  $\mathcal F =\{f_i\}_{i\in I}$ be a subordinate partition of unity. Then,
	\begin{equation}\label{eq:c0_bound_d}
	\max_{i,j\in  I} \|\{f_i,f_j\}\| \geq \frac{1}{2d^2\cdot e(\cU)},
	\end{equation}
	where $d$ is the degree of the cover $\mathcal{U}$.
\end{cor}

\begin{rem} 
	The dependence on $d$ in the bound presented in Corollary~\ref{cor:c0_bound} is optimal. To see this, take any open displaceable cover  $\mathcal{U}=\{U_i\}_{i\in I}$ of $M$ and a subordinate partition of unity $\mathcal{F} = \{f_i\}_{i\in I}$, and denote 
	\begin{equation*}
	b(\mathcal{F}):= \max_{i,j\in I} \|\{f_i,f_j\}\|.
	\end{equation*} 
	We have $b(\mathcal F)>0$ (by the nondisplaceable fiber theorem \cite{entov2004quasi}, or by Corollary~\ref{cor:c0_bound}). For every $m\in\N$ let $\mathcal{U}^m:=\{U_i,\dots,U_i\}_{i\in I}$ be the cover obtained by taking $m$ copies of each set in $\mathcal{U}$ (i.e., $\mathcal{U}^m$ contains $|I|\cdot m$ sets and is of degree $d\cdot m$, where $d$ is the degree of $\mathcal U$). Consider the subordinate partition $\mathcal{F}^m:=\{\frac{1}{m}f_i,\dots,\frac{1}{m}f_i\}_{i\in I}$. Then  
	\begin{equation*}
	b(\mathcal{F}^m) = \max_{i,j\in I} \|\{\frac{1}{m}f_i,\frac{1}{m}f_j\}\| = \frac{1}{m^2} b(\mathcal{F})
	\end{equation*}
	decays quadratically in the degree of the cover. 
\end{rem}

\subsection{Bounds in higher dimensions.}
From Corollary~\ref{cor:pb_bound_gen} one can conclude that when the sets in $\mathcal{U}$ are small, $pb(\mathcal U)$ must be large. The following proposition was explained to us by Leonid Polterovich and shows that this is true in higher dimensions as well. 

\begin{prop}\label{pro:pb_c0}
	Let $(M,\omega)$ be any closed symplectic manifold of dimension $2n$ and let $\rho$ be any Riemannian metric on $M$. For any $\epsilon>0$, let $\mathcal{U}^\epsilon$ be a finite cover of $M$ by open subsets of diameter at most $\epsilon$ (with respect to the metric $\rho$). Then, 
	\begin{equation}\label{eq:pb_c0}
	pb(\mathcal{U}^\epsilon) \underset{\epsilon\rightarrow 0}{\longrightarrow} \infty.
	\end{equation}
\end{prop}

One should expect the rate of convergence in Proposition~\ref{pro:pb_c0} to be quadratic in $1/\epsilon$. This is due to the fact that the Poisson bracket is homogeneous of degree 2 with respect to composition with homothetic transformations of $\R^{2n}$: Given smooth functions $g,h:\R^{2n}\rightarrow\R$, and a homothetic transformation $\psi_c:\R^{2n}\rightarrow\R^{2n}$, $\psi_c(x)=c\cdot x$ for some $c>0$, 
\begin{equation*}
\{g\circ\psi_c, h\circ\psi_c\}(x) =c^2\{g,h\}(cx).
\end{equation*} 
The next theorem shows that this is indeed the case.
\begin{thm}\label{thm:rate}
	Consider the setting of Proposition~\ref{pro:pb_c0} and assume in addition that $\rho$ is compatible with $\omega$. Then, there exists a constant $c=c(n)>0$ depending only on the dimension, and a constant  $\delta = \delta(M,\omega,\rho)>0$, depending on the symplectic manifold $(M,\omega)$ and the metric $\rho$, such that for every $\epsilon\leq \delta$,
	\begin{equation}
	pb(\mathcal U^\epsilon)\geq \frac{c}{\epsilon^2}.
	\end{equation}
\end{thm}\

\subsection{Acknowledgements}
We are deeply grateful to Fedor Nazarov, without whose help this paper would not have been written. He explained to us his proof of a preliminary version of the statement appearing in Corollary~\ref{cor:c0_bound}, and his ideas have had a significant impact on the paper. Unfortunately, he decided not to be a coauthor of the paper. 

We also thank Leonid Polterovich and Misha Sodin for fruitful discussions. Shira Tanny extends her special thanks to Leonid Polterovich for his mentorship and guidance.

L.B. was partially supported by ISF Grants 1380/13 and 2026/17, by the ERC Starting Grant 757585, and by the Alon Fellowship. A.L. was partially supported by ERC Advanced Grant~692616 and ISF Grants~1380/13 and 382/15. S.T. was partially supported by ISF Grants 178/13, 1380/13, and  2026/17. 


\section{Essential sets and Poisson bracket.}\label{sec:essential_sets}

Let $M$ be a closed connected surface, endowed with an area form $\omega$. For any smooth function $f$, we denote by $cp(f)$ the set of its critical points and by $cv(f) = f(cp(f))$ the set of its critical values. 
Our first lemma explains the relation between the $L^1$ norm of the Poisson brackets of two functions and intersections of their level sets.

\begin{lemma}\label{lem:pb_to_intersections}
	Let $f,g:M\rightarrow\R$ and denote $\Phi:=(f,g):M\rightarrow\R^2$. Consider the function $K:\R^2\rightarrow\R\cup\{\infty\}$ defined by $K(s,t):= \#(f^{-1}(s)\cap g^{-1}(t)) = \#\Phi^{-1}(s,t)$, then for any Lebesgue measurable set $\Omega\subset \R^2$, 
	\begin{equation}\label{eq:pb_to_intersections}
	\int_{\Phi^{-1}(\Omega)} |\{f,g\}|\ \omega = \int_\Omega K(s,t)\ ds\ dt.
	\end{equation}
\end{lemma}
Note that the integral on the left-hand side is taken with respect to the volume density given by $\omega$. For the proof of the lemma, see Appendix \ref{app:lemma-intgeom-proof}.

Lemma~\ref{lem:pb_to_intersections} suggests that one can estimate the $L^1$ norm of the Poisson bracket of two functions by counting intersections of their level sets. It turns out that when $f:=f_i$ corresponds to an essential set $U_i\in\mathcal U$, one can bound from below the number of intersections of level sets of $f_i$ and level sets of any other function $f_j$ from the partition of unity. For a more formal description we need to present some notations. 
Given an open cover $\mathcal U= \{U_i\}_{i\in I}$ and a subordinate partition of unity $\mathcal F=\{f_i\}_{i\in I}$, denote  
\begin{equation}
	U_i(t):=\{x\in M: f_i(x)>t\},
\end{equation} 
for $i\in I$, $t\geq0$.
Clearly, for any such $t$, $U_i(t)\subset U_i$. Moreover, the boundary of $U_i(t)$ is contained in the $t$-level set of $f_i$, namely, $\partial U_i(t) \subset \{x\in M: f_i(x)=t\}$. For a subset $U\subset M$, we denote by $U^c:=M\setminus U$ its complement. The following definitions will be useful:

\begin{defin} \label{def:ps-bdry}
Let $ M $ be a smooth closed surface, and let $ V \subset M $ be an open {(or closed)} set. We say that $ V $ has a piecewise smooth boundary if $ \partial V $ is a finite union of disjoint curves $ \Gamma_1, \ldots, \Gamma_m $, such that each $ \Gamma_j $ is a simple, closed, piecewise smooth and regular curve. 
\end{defin}

\begin{defin} \label{def:enclosing-disc}
	Let $V\subset M$ be an open (or closed) connected subset with a piecewise smooth boundary, which is contained in a topological disc of area less than $ area(M)/2 $. There exists a unique connected component of $ M \setminus V $ of area greater than $ area(M)/2 $. The {\it enclosing disc} of $ V $ is by definition the complement of this connected component, and it is denoted by $ \tilde V $ (see Figure \ref{fig:enclosing_disc} for an example).
\end{defin} 

\begin{figure}
	\centering
	\includegraphics[scale=0.6]{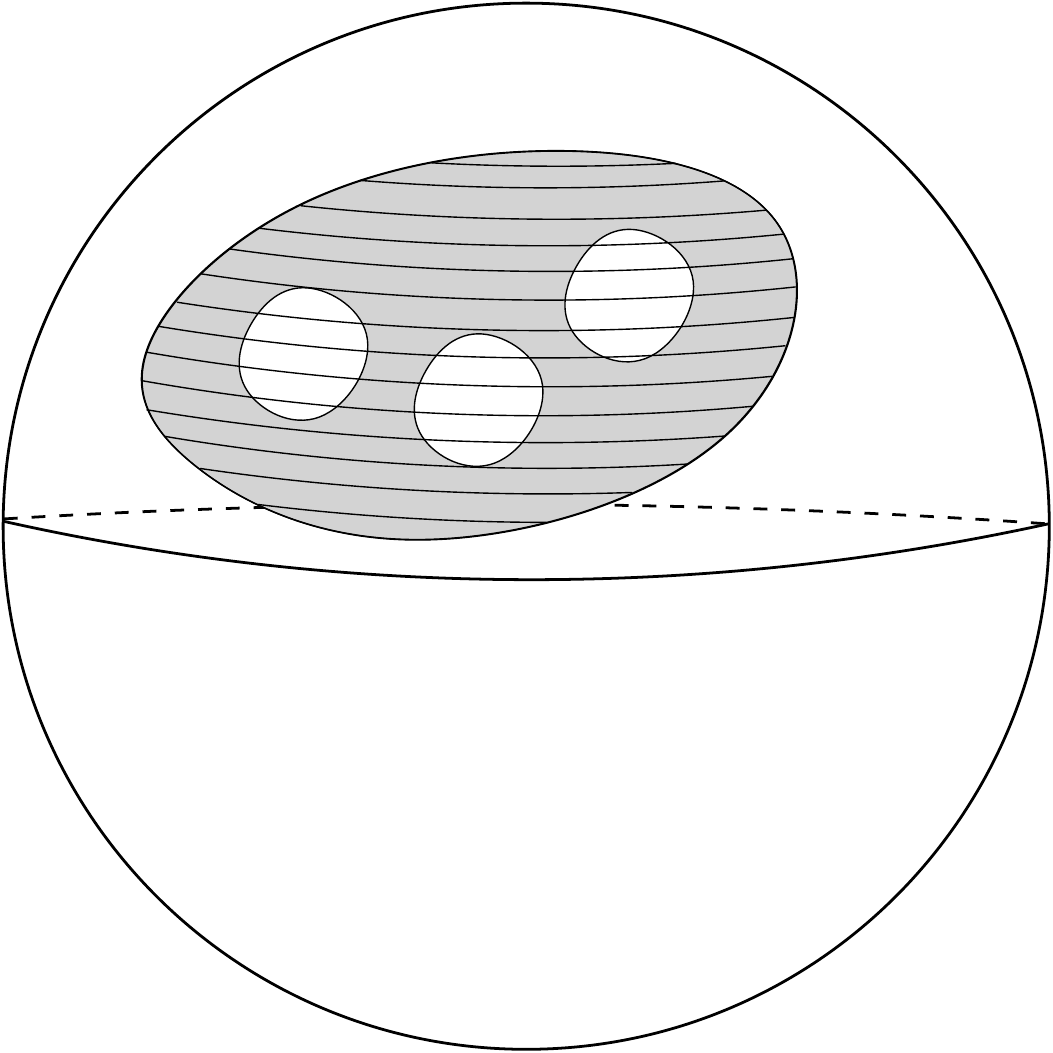}
	\caption{\small{The grey domain on the sphere and its enclosing disc (whose area is hatched).}}
	\label{fig:enclosing_disc}	
\end{figure}

\begin{rem}\label{rem:enlosing-disc}\
\begin{itemize}
 \item For any subset $V\subset M$ as in Definition \ref{def:enclosing-disc}, we have $\partial \tilde V\subset \partial V$.
 \item Let $V\subset M$ be a subset as in Definition \ref{def:enclosing-disc}. Then its enclosing disc $\tilde V$ is  the open (respectively, closed) topological disc of minimal area that contains it. In particular, if $ U $ is an open topological disc of area less than $ area(M)/2 $ which compactly contains $ V $, then $ U \supset \tilde V $.
\end{itemize}
\end{rem}

\begin{proof}[Proof of Theorem~\ref{thm:essential}]
In the following we prove that if $U_i$ is essential then 
\begin{equation}\label{eq:1ess_bnd}
\sum_{j\in I}\int_M |\{f_i,f_j\}|\ \omega\geq 1.
\end{equation}
Summing (\ref{eq:1ess_bnd}) over all $i\in I_{ess}(\mathcal U)$ yields (\ref{eq:ess_bnd1}). To conclude (\ref{eq:ess_bnd2}), apply (\ref{eq:1ess_bnd}) to the essential set of minimal area and notice that $$\sum_{j\in I}\int_M |\{f_i,f_j\}|\ \omega = \sum_{j\in I}\int_{U_i} |\{f_i,f_j\}|\ \omega\leq area(U_i)\cdot \max_M\sum_{j,k\in I} |\{ f_k,f_j\}|.$$

We turn to prove (\ref{eq:1ess_bnd}). Fix $i\in I_{ess}$, then there exists a point $z_i\in U_i$ such that for all $j\neq i$, $z_i\notin U_j$. Since all functions but $f_i$ vanish at $z_i$, we conclude that $f_i(z_i) = 1 $ and hence $z_i\in U_i(s)$ for all $s\in (0,1)$. For a regular value $ s \in (0,1) $ of $ f_i $, denote by $V_i(s)$ the connected component of $ U_i(s) $ that contains $z_i$, and by $\tilde V_i(s)$ the enclosing disc of $V_i(s)$. We have $\partial \tilde V_i(s)\subset \partial V_i(s)$. Denote 
\begin{equation}
\gamma^s:= \partial \tilde V_i(s),
\end{equation}
then $\gamma^s$ is connected and is contained in the level set $\{f_i =s\}$. 
For every regular value $s\in (0,1)$ of $ f_i $, fix $y^s\in \gamma^s$ and for each $j\neq i$ denote $t_j^s:= f_j(y^s)\in \R$. Fix $j\neq i$, and let $t\in (0,t_j^s)$ be a regular value of $ f_j $. We have $y^s\in U_j(t)$, since $f_j(y^s) = t_j^s>t$. Denote by $D_j(t)$ the closure of the connected component of $U_j(t)$ that contains $y^s$, and denote by $\tilde D_j(t)$ the enclosing disc of $ D_j(t) $. Then, $\partial \tilde D_j(t)\subset \partial D_j(t)\subset \{f_j = t\}$. See Figure~\ref{fig:essential_proof} for a demonstration of this setting.
\begin{figure}
	\centering
	\includegraphics[scale=0.7]{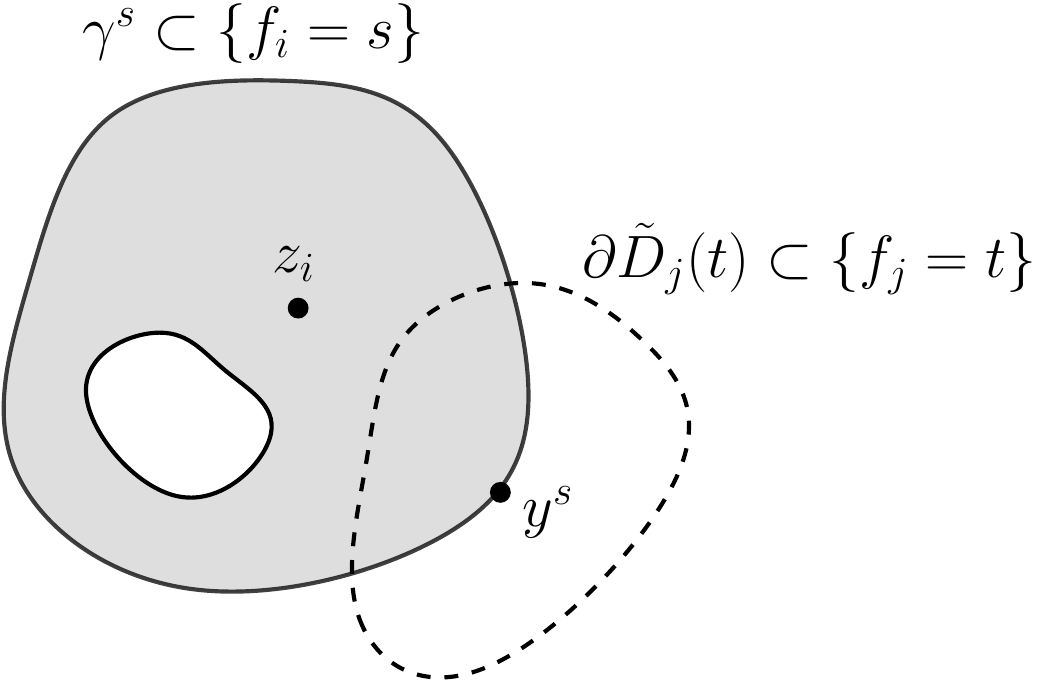}
	\caption{\small{An example for the setting described in the proof of Theorem~\ref{thm:essential}. In this example, the gray region is $V_i(s)$, the solid lines are the $s$-level set of $f_i$, and the outer component is $\gamma^s$. The dashed line is the boundary of $ \tilde D_j(t) $, which is a part of the $t$-level set of $f_j$, for some $t<f_j(y^s)$.  }}
	\label{fig:essential_proof}	
\end{figure}
We claim that $\gamma^s$ has at least two points of intersection with $\partial \tilde D_j(t)$. Since the interior of $\tilde D_j(t)$ intersects $\gamma^s$ (as they both contain $y^s$), it is enough to show that $\gamma^s$ is not contained in $\tilde D_j(t)$. Recalling that $\gamma^s$ is the boundary of  $\tilde V_i(s)$, this is equivalent to showing that both $\tilde V_i(s)$ and its complement  $\tilde V_i(s)^c$ are not contained in  $\tilde D_j(t)$. 
Recall that $U_j$ is a topological disc containing $D_j(t)$, and hence $\tilde D_j(t)\subset U_j$. The topological disc $\tilde V_i(s)$ contains $z_i\notin U_j$ and thus is not contained in $U_j$. In particular, we conclude that $V_i(s)$ is not contained in $\tilde D_j(t)$. Finally, to show that $\tilde V_i(s)^c\nsubseteq \tilde D_j(t)$, recall that $\tilde V_i(s)\subset U_i$ (since $U_i$ is a topological disc containing $V_i(s)$) and therefore   
$$
area (\tilde D_j(t))\leq area(U_j)<\frac{area(M)}{2}< area(U_i^c)\leq area(\tilde V_i(s)^c).
$$ 
This implies that $\tilde V_i(s)^c\nsubseteq \tilde D_j(t)$ and hence we conclude that $\gamma^s$ intersects $\partial \tilde D_j(t)$ at least twice. Since $\gamma^s\subset \{f_i = s\}$ and $\partial \tilde D_j(t)\subset \{f_j=t\}$, we have $\#\{f_i = s\}\cap\{f_j = t\}\geq 2$ for any regular value $s\in(0,1)$ of $ f_i $, and any regular value $t\in (0,t_j^s)$ of $ f_j $. Putting $K_{ij}(s,t):= \#\{f_i = s\}\cap\{f_j = t\}$ and applying Lemma~\ref{lem:pb_to_intersections} to $f_i$ and $f_j$ with $\Omega:= \{(s,t): t\in (0,t_j^s), s\in(0,1)\}$ we obtain 
\begin{eqnarray*}
\int_{M} |\{f_i,f_j\}|\ \omega &\geq&\int_{\Phi^{-1}(\Omega)} |\{f_i,f_j\}|\ \omega\\
&=&\int_\Omega K_{ij}(s,t)\ ds\ dt\\
&\geq& \int_0^1 \int_0^{t_j^s} 2 \ dt\ ds = 2\int_0^1 t_j^s ds. 
\end{eqnarray*}
Now, recalling that $t_j^s = f_j(y^s)$, and summing the above inequality over all $j\neq i$ we get
\begin{eqnarray*}
\sum_{j\in I} \int_{M} |\{f_i,f_j\}|\ \omega &\geq& 2\sum_{j\neq i}\int_0^1 f_j(y^s)\ ds\\
&=&  2\int_0^1 \sum_{j\neq i} f_j(y^s)\ ds = 2\int_0^1 1-f_i(y^s)\ ds.
\end{eqnarray*}
Since we chose $y^s\in \gamma^s\subset \{f_i = s\}$, we have $f_i(y^s) = s$ and thus
\begin{equation*}
\sum_{j\in I} \int_{M} |\{f_i,f_j\}|\ \omega \geq 2\int_0^1 1-s\ ds = 2\cdot\frac{1}{2} = 1.
\end{equation*}
\end{proof}

\begin{rem} \label{rem:not-ness-discs}
	In Theorem \ref{thm:essential} we assume that the covering sets $ U_i $ are topological discs. However, 
	when an open cover $\cU = \{ U_i \}_{i \in I} $ does not necessarily consist of topological discs, but the covering sets have piecewise smooth boundary, 
	then we can pass to a cover by topological discs in two steps. 
	
	First, consider the collection $ \cV = \{ V_j \}_{j \in J} $ of all connected components of all the $ U_i $'s. Given any partition of 
	unity $ \cF = \{ f_i \}_{i \in I} $, subordinate to $ \cU $, we naturally get a 
	partition of unity $ \cG = \{ g_j \}_{j \in J} $   subordinate to $ \cV $, as follows: for every $ V_j $ being a connected component 
	of $ U_i $, we set $ g_j = f_i \id_{V_j}$, where $ \id_{V_i} $ is the characteristic function of $ V_j $ on $ M $.
	We moreover have $ \sum_{i,j\in I} | \{f_i,f_j\} | = \sum_{i,j \in J} | \{g_i,g_j\} | $. This reduces proving estimates  $(\ref{eq:ess_bnd1})$ 
	and $(\ref{eq:ess_bnd2})$ from Theorem~\ref{thm:essential} for the cover $ \cU $, to 
	proving them for the cover $ \cV $. Of course, if the covering sets $ U_i $ are connected from the beginning, the cover $ \cV $ is the same as $ \cU $.
	
	Second, denoting by $ \tilde V_j $ the enclosing disc of $ V_j $, for each $ j $, we get a cover $ \tilde \cV = \{ \tilde V_j \}_{j\in J} $ by displaceable open topological discs, and the partition of unity $ \cG = \{ g_j \}_{j\in J} $ clearly 
	subordinate to $ \tilde \cV $ as well.  Therefore any lower bound for the latter cover will also hold for $\mathcal U$. However, one should notice that when applying the first part of Theorem~\ref{thm:essential} to such a 
	general cover $ \cU $ by open sets with piecewise smooth boundaries (not necessarily by topological discs), the bound will depend on the number of essential sets in $\tilde \cV$:
	\begin{equation}
	\int_M \sum_{i,j} |\{f_i,f_j\}|\ \omega\geq |I_{ess}(\tilde{\mathcal V})|.
	\end{equation}
	The second part of Theorem~\ref{thm:essential} can be written in terms of the displacement energy of sets in $\mathcal V$. Indeed, by Remarks~\ref{rem:energy_vs_area} and~\ref{rem:enlosing-disc},  $e(V) = area(\tilde V)$. Applying the second part of Theorem~\ref{thm:essential} to $\tilde{\mathcal V}$ yields
	\begin{equation}
	\max_M\sum_{i,j} |\{f_i, f_j\}| \geq \frac{1}{\min_{\ell\in I_{ess}(\tilde{\mathcal V})} e(V_\ell)}.
	\end{equation}	
\end{rem}

\section{Bounds for general covers.}\label{sec:gen_covers}

In the general case, estimating the number of intersections of level sets is more complicated.  
\begin{defin}
	Two covers  $\mathcal U=\{U_i\}_{i\in I}$, $\mathcal V=\{V_j\}_{j\in J}$ of $M$ are said to be {\it in generic position} if the following triple intersections of boundaries are empty:
	\begin{equation}
	\partial U_i\cap \partial U_k\cap \partial V_j=\emptyset,\ \ \partial U_i\cap\partial V_j\cap \partial V_\ell =\emptyset,
	\end{equation}
	for all $i,k\in I$, $ i \neq k $, and $j,\ell\in J$, $ j \neq \ell $.
\end{defin}
The central lemma in the proof of Theorem~\ref{thm:gen_cover} is the following:
\begin{lemma}\label{lem:bound_no_intersections}
	Let  $\mathcal U=\{U_i\}_{i\in I}$, $\mathcal V=\{V_j\}_{j\in J}$ be finite open covers of $M$ with smooth boundaries, and assume that $\mathcal U$ and $\mathcal V$ are in generic position. Moreover, assume that for some $ 0 < A < area(M)/2 $, each element of $ \cU $ or $ \cV $ is compactly contained inside a topological disc of area not greater than $ A $. Suppose in addition that there exists $L\in\N$ such that for any point $x\in M$, $\#\{i\in I:x\in U_i\}\geq L$ and $\#\{j\in J:x\in V_j\}\geq L$. Then 
	\begin{equation}
	\# \cup_{i,j}(\partial U_i\cap\partial V_j)\geq L^2\cdot \frac{area(M)}{2A}.	
	\end{equation}
\end{lemma}

Let us illustrate the heuristics underlying the proof of Theorem~\ref{thm:gen_cover} before giving the details.
Let $\cU=\{U_i\}_{i\in I}$, $\cV=\{V_j\}_{j\in J}$ be two open covers of $M$ and let $\cF=\{f_i\}_{i\in I}$, $\cG=\{g_j\}_{j\in J}$ be subordinate partitions of unity, as in the theorem. In the light of Lemma~\ref{lem:pb_to_intersections}, we wish to estimate the number of intersections of level sets. 
Fix $L\in\N$ sufficiently large and denote $U_{i,k}=\big\{f_i>\frac{k}{L}\big\}$, $V_{j,\ell}:=\big\{g_j>\frac{\ell}{L}\big\}$, where $ k $ and $ \ell $ are positive integers. Then the boundaries of $U_{i,k}$ and $V_{j,\ell}$ are contained in level sets of $f_i$, $g_j$ respectively. Given $x\in M$, let us estimate the number of sets in $\{U_{i,k}\}_{i,k}$ containing $x$. For fixed $i\in I$, 
\begin{equation*}
\#\{k: x\in U_{i,k}\} = \{k: f_i(x)>k/L\} \geq Lf_i(x) -1.
\end{equation*}
Therefore, the number of sets $U_{i,k}$ containing $x$ is at least $\sum_i(Lf_i(x) - 1) = L-|I|$. Similarly, one can show that the number of sets $V_{j,\ell}$ containing $x$ is at least $L-|J|$. In particular, when $L$ is sufficiently large, $\{U_{i,k}\}_{i,k}$ and $\{V_{j,\ell}\}_{j,\ell}$ are open covers of $M$, that satisfy the conditions of Lemma~\ref{lem:bound_no_intersections} for $\hat L:= L-|I|-|J|$ (namely, every point in $M$ is contained in at least $\hat L$ sets). Applying Lemma~\ref{lem:bound_no_intersections} to the covers $\{U_{i,k}\}_{i,k}$, $\{V_{j,\ell}\}_{j,\ell}$ we obtain 
\begin{equation*}
\sum_{i,k,j,\ell} \#\partial U_{i,k}\cap\partial V_{j,\ell} \geq \hat L^2\cdot \frac{area(M)}{2A}.
\end{equation*}
On the other hand, one expects that in a generic situation, given $i,j$ and sufficiently large $L$, the sum $\frac{1}{L^2} \sum_{k,\ell} \#\partial U_{i,k}\cap\partial V_{j,\ell} =\frac{1}{L^2} \sum_{k,\ell} \#\big(f_i^{-1}(\frac{k}{L})\cap g_j^{-1}(\frac{\ell}{L})\big)$ will approximate the integral of $K_{ij}(s,t):= \#\big(f_i^{-1}(s)\cap g_j^{-1}(t)\big)$. Using Lemma~\ref{lem:pb_to_intersections} and taking the limit $L\rightarrow\infty$ we obtain
\begin{equation*}
\sum_{i,j} \int_M |\{f_i,g_j\}|\ \omega \geq \lim_{L\rightarrow\infty} \frac{\hat L^2}{L^2}\cdot \frac{area(M)}{2A}, 
\end{equation*}
which implies Theorem~\ref{thm:gen_cover}.

Now let us pass to the actual proofs. We will need the following definition:
\begin{defin}
	Let $ \gamma_1, \ldots, \gamma_m \subset M $ be a finite collection of smooth regular curves with a finite number of mutual intersection points. Denote $\Gamma = \gamma_1 \cup \cdots \cup \gamma_m $. 
	\begin{itemize}
		\item A connected component of the complement $M\setminus \Gamma$ is called a {\it face} of $\Gamma$.
		\item A point $v\in\Gamma$ that lies in the intersection of two (or more) curves is called a {\it vertex} of $\Gamma$.
		\item $\Gamma$ is called an {\it $A$-division of $M$}, if every face of $\Gamma$ has a piecewise smooth boundary (as in Definition \ref{def:ps-bdry}) and is compactly contained in an open topological disc of area not greater than $A$.
	\end{itemize}
\end{defin}

\begin{lemma}\label{lem:divisions_intersection}
	Let $\Gamma,\ \Gamma'\subset M$ be $A$-divisions of $M$ for some $A < area(M)/2$, and assume that no vertex of $\Gamma$ lies on $\Gamma'$ and vise versa. Then, 
	\begin{equation}
	\# (\Gamma\cap\Gamma')\geq \frac{area(M)}{2A}.
	\end{equation}	
\end{lemma}

\begin{proof}
	First, let us show that by removing parts from $\Gamma$ and $\Gamma'$, we may assume that their faces are open topological discs. The fact that faces 
	of $\Gamma$, $\Gamma'$ are compactly contained in open topological discs of area not greater than $A$ will guarantee that  $\Gamma$, $\Gamma'$ will remain 
	$A$-divisions after removing these parts. More formally, let $P\subset M\setminus\Gamma$ be a face of $ \Gamma $, then it is compactly contained in an open 
	topological disc of area not greater than $A$. Let $\tilde P\supset P$ be the enclosing disc of $ P$. Then, $\partial \tilde P\subset \partial P\subset \Gamma$  and hence, 
	removing $\Gamma\cap \tilde P$ from $\Gamma$, we obtain that $\tilde P$ is a face of $\Gamma$ which is an open topological disc with piecewise 
	smooth boundary (see Figure~\ref{fig:fixing_Gamma}). Moreover, since $ P $ is compactly contained in a topological 
	disc of area not greater than $A$, so is its enclosing disc $\tilde P$. Therefore, $\Gamma$ remains an $A$-division after removing $\Gamma\cap \tilde P $.
	\begin{figure}
		\centering
		\includegraphics[scale=0.7]{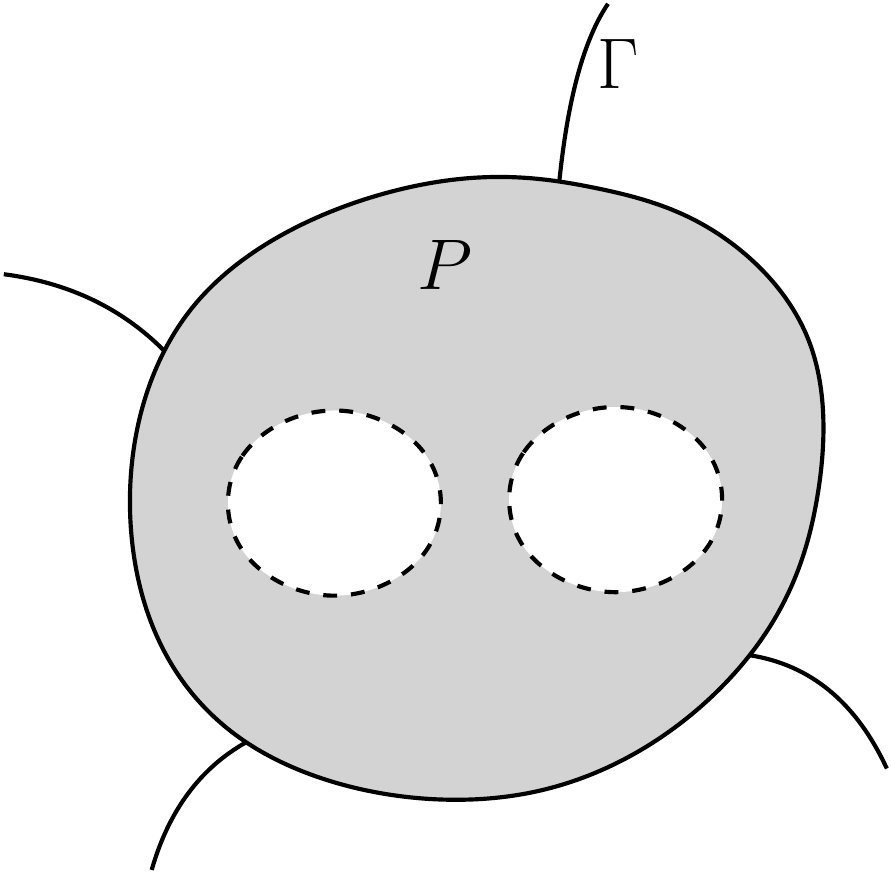}
		\caption{\small{In this example the dashed lines are removed from $\Gamma$. }}
		\label{fig:fixing_Gamma}	
	\end{figure}
	
	Having this assumption we turn to bound the number of intersections of $\Gamma$ and $\Gamma'$. We say that a face $G$ of $\Gamma$ is {\it maximal} if it is not properly contained in any face of $\Gamma'$. Defining similarly maximality of faces of $\Gamma'$, we observe that any non-maximal face of $\Gamma$ is contained in a maximal face of $\Gamma'$. Therefore, the union of maximal faces of both $\Gamma$, $\Gamma'$ covers $M$ up to a subset of area zero. Let us assume, without loss of generality, that the maximal faces of $\Gamma$ cover at least half the area of $M$. Then, since each face has area smaller than $A$, the number of maximal faces of $\Gamma$ is at least $\frac{area(M)}{2A}$. Our next goal is to show that the boundary of every maximal face of $\Gamma$ intersects $\Gamma'$ at least twice. Together with the fact that any intersection point of $\Gamma$ and $\Gamma'$ lies on the boundary of exactly two faces of $\Gamma$ (due to our assumption, that the intersection points are not vertices) this will conclude the proof. 
	
Let $G$ be a maximal face of $\Gamma$, then, there exists a face $G'$ of $\Gamma'$ that intersects the boundary of $G$, $\partial G\cap G'\neq\emptyset$ (otherwise $ \partial G \subset \Gamma' $, in particular $ \# (\Gamma \cap \Gamma') = \infty $, and we are done). We also claim that $ \partial G \cap (M \setminus \overline{G'}) \neq \emptyset $. Indeed, otherwise we have $ \partial G \subset \overline{G'} $, and since $ M \setminus \overline{G'} $ is connected (recall that $ G' \subset M $ is an open topological disc with a piecewise smooth boundary), we have either $ G \supset M \setminus \overline{G'} $ or $ G \subset \overline{G'} $. 
The first option is impossible since $ area(G), area(G') < A < \frac{area(M)}{2} $, and in the second option we get $ G \subset G' $ (since $ G' $ has a piecewise smooth boundary, $ \partial G' $ does not contain interior points of $ \overline{G'} $) which contradicts the maximality of $ G $.	

Hence we conclude that $ \partial G \cap G' \neq \emptyset $ and $ \partial G \cap (M \setminus \overline{G'}) \neq \emptyset $. Since the boundary $ \partial G $ is a simple closed curve, we get $ \# (\partial G \cap \partial G') \geq 2 $.
	
\end{proof}

\begin{proof}[Proof of Lemma~\ref{lem:bound_no_intersections}]
	
By our assumptions, the covers $ \cU $ and $ \cV $ are in generic position. Therefore, we can slightly enlarge the $ U_i $'s and $ V_j $'s, to obtain:

\begin{itemize}
\item[(a)] After the perturbation, $ \partial U_i $ and $ \partial U_j $ intersect transversally for all $ i,j \in I $, $ i \neq j $, and $ \partial V_i $ and $ \partial V_j $ intersect transversally for all $ i,j \in J $, $ i \neq j $.
\item[(b)] The perturbation did not change the intersection points of $ \partial U_i $ with $ \partial V_j $ for every $ i \in I $ and $ j \in J $. In particular, the covers $ \cU $ and $ \cV $ remain to be in generic position after the perturbation.
\item[(c)]  After the perturbation, each of the $ U_i $'s and $ V_j $'s is still compactly contained in a topological disc of area not greater than $ A $.

\end{itemize}
	
In view of that, without loss of generality we can assume from the beginning that the above property (a) is satisfied. Moreover, for the sake of convenience we assume that $ I = \{1,2, \ldots, | I | \} $ and $ J = \{ 1,2,\ldots, |J| \} $.
	
	Now let $\alpha\in S_I$, $\beta\in S_J$ be permutations on the elements of $I,J$ respectively, and consider the unions of curves defined by
	\begin{eqnarray*}
		\Gamma_\alpha &:=& \bigcup_{i\in I} \left(\partial U_{\alpha(i)}\cap U_{\alpha(i-1)}^c\cap\cdots \cap U_{\alpha(1)}^c\right),\\
		\Gamma_\beta' &:=& \bigcup_{j\in J} \left(\partial V_{\beta(j)}\cap V_{\beta(j-1)}^c\cap \cdots \cap V_{\beta(1)}^c\right).
	\end{eqnarray*}
	Let us show that $\Gamma_\alpha$ is an $A$-division of $M$. First, by the property (a), each connected component of $ M \setminus \Gamma_\alpha $ is an open set with a piecewise smooth boundary. Let $P\subset M\setminus\Gamma_\alpha$ be a connected component and assume for the sake of contradiction that $P$ is not compactly contained in any topological disc of area not greater than $A$. Notice that this assumption implies that $P\nsubseteq U_i$ for all $i$, since every set $U_i$ is compactly contained in a topological disc of area not greater than $A$. We show by induction on $i\in I$ that in this case $P\subset U_{\alpha(1)}^c\cap\cdots\cap U_{\alpha(i)}^c$ for all $i$, which immediately leads to a contradiction, as $\cap_{i\in I}U_{\alpha(i)}^c=\emptyset$. Starting with $i=1$, notice that $\partial U_{\alpha(1)} \subset \Gamma_\alpha$. Therefore, $P\cap \partial U_{\alpha(1)}\subset P\cap\Gamma_\alpha=\emptyset$, and since $P\nsubseteq U_{\alpha(1)}$, we conclude that $P\subset U_{\alpha(1)}^c$. Assuming $P\subset U_{\alpha(1)}^c\cap\cdots\cap U_{\alpha(i-1)}^c$, let us show that $P\subset U_{\alpha(i)}^c$. Indeed, since $\partial U_{\alpha(i)}\cap U_{\alpha(i-1)}^c\cap\cdots\cap U_{\alpha(1)}^c\subset \Gamma_\alpha$ and  $P\subset U_{\alpha(1)}^c\cap\cdots\cap U_{\alpha(i-1)}^c$ we conclude that $P\cap\partial U_{\alpha(i)}\subset P\cap \Gamma_\alpha=\emptyset$. Together with the fact that   $P\nsubseteq U_{\alpha(i)}$, this implies $P\subset U_{\alpha(i)}^c$ as required. 
	
	Similarly, one can show that $\Gamma_\beta'$ is also an $A$-division of $M$. The fact that the covers $\mathcal U$ and $\mathcal V$ are in generic position guarantees that no vertex of $\Gamma_\alpha$ lies on $\Gamma_\beta'$ and vice versa.
	Therefore, we may apply Lemma~\ref{lem:divisions_intersection} and conclude a lower bound for the number of intersection points of $\Gamma_\alpha$ and $\Gamma_\beta'$:
	\begin{equation}\label{eq:Gamma_bnd}
	\#\Gamma_\alpha\cap\Gamma_\beta'\geq \frac{area(M)}{2A}\ .
	\end{equation} 
	Clearly, $\Gamma_\alpha\subset \cup_{i\in I} \partial U_i$ and $\Gamma_\beta'\subset\cup_{j\in J} \partial V_j$, and hence $\Gamma_\alpha\cap\Gamma_\beta'\subset \cup_{i,j}(\partial U_i\cap\partial V_j)$. Take a point $x\in \cup_{i,j}(\partial U_i\cap\partial V_j)$ and let us count the number of permutations $\alpha\in S_I$, $\beta\in S_J$ for which $x\in \Gamma_\alpha\cap\Gamma_\beta'$. Let $i\in I$ such that $x\in \partial U_i$, then $x\in \Gamma_\alpha$ only if $\alpha^{-1}(i)<\alpha^{-1}(k)$ for any $k\in I $ such that $x\in U_k$. By our assumption, the number of indices $k\in I$ for which $x\in U_k$ is at least $L$. By symmetry reasons, the number of permutations $\sigma =\alpha^{-1}$ for which $\sigma(i)<\sigma(k)$ for at least $L$ indices $k\in I$ is at most $|I|!/(L+1)$. Similarly, the number of permutations $\beta$ for which $x\in\Gamma_\beta$ is at most $|J|!/(L+1)$. As a consequence, the number of intersection points in $\cup_{i,j}(\partial U_i\cap\partial V_j)$ can be bounded by averaging inequality~(\ref{eq:Gamma_bnd}) over all permutations  $\alpha\in S_I$ and $\beta\in S_J$:
	\begin{eqnarray*}
		\#\cup_{i,j}(\partial U_i\cap\partial V_j) &\geq& \frac{L+1}{|I|!} \sum_{\alpha\in S_I} \frac{L+1}{|J|!} \sum_{\beta\in S_J} \#\Gamma_\alpha\cap\Gamma_\beta'\\
		&\geq& (L+1)^2\cdot \frac{1}{|I|!} \sum_{\alpha\in S_I} \frac{1}{|J|!} \sum_{\beta\in S_J} \frac{area(M)}{2A}\\
		&=& (L+1)^2\cdot \frac{area(M)}{2A} > L^2\cdot \frac{area(M)}{2A}\ .
	\end{eqnarray*}
\end{proof}

\begin{proof}[Proof of Theorem~\ref{thm:gen_cover}]
	Given $L\in \N$ sufficiently large, we wish to use the functions $\{f_i\}_{i\in I}$, $\{g_j\}_{j\in J}$ to construct covers that satisfy the assumptions of Lemma~\ref{lem:bound_no_intersections}. For every $i\in I$ and $j\in J$ pick $m_i,n_j\in \N$ such that $\frac{m_i}{L}>\max_M f_i$ and $\frac{n_j}{L}>\max_M g_j$. For any $i\in I$ and $1\leq k\leq m_i$ consider the interval $\cI_{i,k}:=\big[\frac{k-1}{L},\frac{k}{L}\big]$ and denote by $s_{i,k}\in\cI_{i,k}$ an independent variable. We think of $s_{i,k}$ as representing a value of the function $f_i$. We equip the interval $\cI_{i,k}$ with the normalized Lebesgue measure $\mu_{i,k}:= Lds_{i,k}$. Similarly, for $j\in J$ and $1\leq \ell\leq n_j$, consider the interval $\cJ_{j,\ell}:= \big[\frac{\ell-1}{L}, \frac{\ell}{L} \big]$ and let $t_{j,\ell}\in \cJ_{j,\ell}$ be an independent variable. We think of $t_{j,\ell}$ as representing a value of the function $g_j$, and equip $\cJ_{j,\ell}$ with the normalized Lebesgue measure $\nu_{j,\ell}:=Ldt_{j,\ell}$. Denote by $\cC:= \prod_{i\in I}\prod_{1\leq k\leq m_i} \cI_{i,k}$, $\cD:= \prod_{j\in J}\prod_{1\leq \ell\leq n_j} \cI_{j,\ell}$  the products of the intervals, then $\cC\subset \R^m$, $\cD\subset \R^n$ for $m:=\sum_{i\in I} m_i$ and $n:=\sum_{j\in J} n_j$. For  ${\bf s}:= (s_{i,k})_{i,k}\in\cC$ and ${\bf t}:= (t_{j,\ell})_{j,\ell}\in\cD$, consider the open sets 
	\begin{eqnarray*}
		U_{i,k}^{\bf s} &:=& U_{i,k}(s_{i,k}) = \{f_i>s_{i,k}\}\quad  1\leq k\leq m_i,\ i\in I,\\
		V_{j,\ell}^{\bf t} &:=& V_{j,\ell}(t_{j,\ell}) = \{g_j>t_{j,\ell}\}\quad\  1\leq \ell \leq n_j,\ j\in J.
	\end{eqnarray*}
	Note that when $L$ is sufficiently large, $\cU^{\bf s}:=\{U_{i,k}^{\bf s}\}_{i,k}$ and $\cV^{\bf t}:= \{V_{j,\ell}^{\bf t}\}_{j,\ell}$ are open covers of $M$, for any ${\bf s}\in \cC$ and ${\bf t}\in \cD$. Let us show that these covers satisfy the assumptions of Lemma~\ref{lem:bound_no_intersections}. Let $x\in M$, then for every $i\in I$,
	\begin{eqnarray*}
		\#\{1\leq k\leq m_i: x\in U_{i,k}^{\bf s}\} &=&\#\{1\leq k\leq m_i: f_i(x)>s_{i,k}\}\\
		&\geq& \#\big\{1\leq k\leq m_i: f_i(x)>\frac{k}{L}\big\}\\
		&\geq& Lf_i(x)-1.
	\end{eqnarray*}
	Therefore, the number of sets in  $\cU^{\bf s}$ covering $x$ is at least $\sum_{i\in I} (Lf_i(x)-1)  = L - |I|>L-|I|-|J|$. Similarly, the number of sets 
	in $\cV^{\bf t}$ covering $x$ is at least $L-|J|>L-|I|-|J|$. 
	In addition, we claim that for almost all ${\bf (s,t)}\in \cC\times \cD$ (namely, except for a set of measure zero) the covers  $\cU^{\bf s}$ and $\cV^{\bf t}$ are in generic position. 
	Indeed, by Sard's theorem, for almost all  ${\bf (s,t)}\in \cC\times \cD$, $(s_{i,k}, t_{j,\ell})$ is a regular value of the map $M\rightarrow\R^2$, $x\mapsto (f_i(x), g_j(x))$ for 
	all $i$, $k$, $j$ and $\ell$. In particular, for such ${\bf (s,t)}$, the boundaries $\partial U_{i,k}^{\bf s}$ and $\partial V_{j,\ell}^{\bf t}$ intersect transversely at a finite 
	number of points. Therefore, by restricting the set of ${\bf (s,t)}$ slightly further, we can guarantee that the covers $\cU^{\bf s}$ and $\cV^{\bf t}$ are in generic 
	position. Recalling that $U_{i,k}^{\bf s} = \{f_i>s_{i,k}\}\subset supp(f_i) \subset U_i$, we conclude that each $ U_{i,k}^{\bf s} $ lies in a topological disc of area not 
	greater than $ A $. Similarly, each $ V_{j,\ell}^{\bf t} $ lies in a topological disc of area not greater than $ A $. This completes the assumptions 
	of Lemma~\ref{lem:bound_no_intersections}, and applying it for $\hat L:=L-|I|-|J|$ and almost every ${\bf (s,t)}$, we obtain 
	\begin{equation*}
	\#\cup_{i,k,j,\ell} (\partial U_{i,k}^{\bf s}\cap \partial V_{j,\ell}^{\bf t})\geq \hat L^2\cdot \frac{area(M)}{2A}. 
	\end{equation*}
	
	Averaging the above inequality over  ${\bf (s,t)}\in \cC\times \cD$ with respect to the normalized product measure $\mu\times\nu$ where 
	$\mu:= \prod_{i,k}\mu_{i,k}$ and $\nu:=\prod_{j,\ell}\nu_{j,\ell}$ we obtain
	\begin{eqnarray}\label{eq:sum_of_intersections}
	\frac{\hat L^2\cdot area(M)}{2A} &\leq& \int \#\cup_{i,k,j,\ell} (\partial U_{i,k}^{\bf s}\cap \partial V_{j,\ell}^{\bf t}) \ d\mu({\bf s})\, d\nu({\bf t}) \nonumber\\
	&\leq&  \int \sum_{i,k,j,\ell} \#(\partial U_{i,k}^{\bf s}\cap \partial V_{j,\ell}^{\bf t}) \ d\mu({\bf s})\, d\nu({\bf t})\nonumber\\
	&=& \sum_{i,k,j,\ell} \int_{\frac{k-1}{L}}^{\frac{k}{L}} \int_{\frac{\ell-1}{L}}^{\frac{\ell}{L}}\#(\partial U_{i,k}^{\bf s}\cap \partial V_{j,\ell}^{\bf t}) \ d\mu_{i,k}(s_{i,k})\, d\nu_{j,\ell}(t_{j,\ell})\nonumber\\
	&=& L^2 \sum_{i,k,j,\ell} \int_{\frac{k-1}{L}}^{\frac{k}{L}} \int_{\frac{\ell-1}{L}}^{\frac{\ell}{L}}\#(\partial U_{i,k}^{\bf s}\cap \partial V_{j,\ell}^{\bf t}) \ ds_{i,k}\, dt_{j,\ell}.
	\end{eqnarray}
		
	For any values of $ s_{i,k} $ and $ t_{j,\ell} $, we have $ \partial U_{i,k}^{\bf s} = \partial \{ f_i > s_{i,k} \} \subset f_i^{-1}(s_{i,k}) $ and $ \partial V_{j,\ell}^{\bf t} = \partial \{ g_j > t_{j,\ell} \} \subset g_j^{-1}(t_{j,\ell}) $. Hence from ($\ref{eq:sum_of_intersections}$) we conclude 
	
         \begin{eqnarray}\label{eq:sum_of_intersections-2}
	\frac{\hat L^2\cdot area(M)}{2A} &=& \frac{(L-|I|-|J|)^2 \cdot area(M)}{2A}\\ &\leq& L^2 \sum_{i,k,j,\ell} \int_{\frac{k-1}{L}}^{\frac{k}{L}} \int_{\frac{\ell-1}{L}}^{\frac{\ell}{L}}\#(f_i^{-1}(s_{i,k}) \cap g_j^{-1}(t_{j,\ell})) \ ds_{i,k}\, dt_{j,\ell}.\nonumber
	\end{eqnarray}
	
	Now we wish to use Lemma~\ref{lem:pb_to_intersections} in order to obtain a lower bound for the Poisson brackets of the functions. Denote $\Phi_{i,j}:=(f_i, g_j):M\rightarrow \R^2$ and set $\Omega_{k,\ell}:=\big(\frac{k-1}{L}, \frac{k}{L}\big)\times \big(\frac{\ell-1}{L}, \frac{\ell}{L}\big)\subset \R^2$.  Applying Lemma~\ref{lem:pb_to_intersections} to each term of the sum in (\ref{eq:sum_of_intersections-2}), we obtain
	\begin{eqnarray*}
		\frac{(L-|I|-|J|)^2  \cdot area(M)}{2A} &\leq&  L^2 \sum_{i,k,j,\ell} 
		\int_{\Omega_{k,\ell}}\#(f_i^{-1}(s_{i,k}) \cap g_j^{-1}(t_{j,\ell})) \ ds_{i,k}\, dt_{j,\ell}\\
		&=& L^2 \sum_{i,k,j,\ell} \int_{\Phi_{i,j}^{-1}(\Omega_{k,\ell})} |\{f_i, g_j\}|\ \omega\\
		&\leq& L^2\sum_{i,j} \int_M  |\{f_i, g_j\}|\ \omega,
	\end{eqnarray*}
	where in the last inequality we use the fact that the domains $\{\Omega_{k,\ell}\}_{k,\ell}$ are disjoint, and so are their pre-images under $\Phi_{i,j}$ for fixed $i,j$ (in fact, it follows from the proof of Lemma~\ref{lem:pb_to_intersections} that this inequality is an equality, since the union of the $ \Omega_{k,\ell} $'s contains $ \Phi_{i,j}(M) $ up to a set of measure zero). 
	We conclude that 
	\begin{eqnarray*}
		\sum_{i,j} \int_M  |\{f_i, g_j\}|\ \omega&\geq& \frac{(L-|I|-|J|)^2}{L^2}\cdot \frac{ area(M)}{2A} \\
		&\underset{L\rightarrow\infty}{\longrightarrow}& \frac{ area(M)}{2A}
	\end{eqnarray*}
	which proves the claim.
\end{proof}

\begin{rem} \label{rem:equivalence}
Let us comment about equivalence of Theorems \ref{thm:gen_cover} and \ref{thm:gen_cover}' (cf. Remark \ref{rem:not-ness-discs}). Due to Remark \ref{rem:energy_vs_area}, in case when the partitions of unity $ \cU $ and $ \cV $ consist of connected open sets, the statement of Theorem \ref{thm:gen_cover}' is equivalent to Theorem \ref{thm:gen_cover}. If, however, not all the elements of the covers $ \cU $ and $ \cV $ are connected, then the statement of Theorem \ref{thm:gen_cover}' still follows from Theorem \ref{thm:gen_cover}. 

Indeed, consider the collection $ \widehat \cU = \{ \widehat U_i \}_{i \in \hat I} $ of all connected components of all the $ U_i $'s, and similarly, consider the collection $ \widehat \cV = \{ \widehat V_j \}_{j \in \hat J} $ of all connected components of all the $ V_j $'s. Given any partition of unity $ \cF = \{ f_i \}_{i \in I} $, subordinate to $ \cU $, and a partition of unity $ \cG = \{ g_j \}_{j \in J} $, subordinate to $ \cV $, we naturally get a partition of unity $ \widehat \cF = \{ \hat f_i \}_{i \in \hat I} $ subordinate to $ \cU $ and a partition of unity $ \widehat \cG = \{ \hat g_j \}_{j \in \hat J} $ subordinate to $ \widehat \cV $, as follows. For every $ \widehat U_k $ being a connected component of $ U_i $, we set $ \hat f_k = f_i \id_{\widehat U_k}$, where $ \id_{\widehat U_k} $ is the characteristic function of $ \widehat U_k $ on $ M $. The description of the partition of unity $ \widehat \cG $ is similar. Of course, in general the covers $ \widehat \cU $ and $ \widehat \cV $ might be infinite, but since the functions $ f_i $ and $ g_j $ have compact support, it follows that $ \hat f_i $ and $ \hat g_j $ are non-trivial only for a finite number of $ i \in \hat I $ and $ j \in \hat J $. We have $ \sum_{i \in \hat I,j \in \hat J} | \{\hat f_i,\hat g_j\} | = \sum_{i \in I,j \in J} | \{f_i,g_j\} | $ on $ M $. This reduces proving the statement of the theorem for the covers $ \cU, \cV $, to proving it for the covers $ \widehat \cU, \widehat \cV $ consisting of connected open sets. 
\end{rem}

\begin{rem}
	In fact, the following more generalized formulation of Theorem~\ref{thm:gen_cover} holds. Let $ M $ be a closed and connected surface endowed 
	with an area form $ \omega $, and let $\{f_i\}_{i\in I}$,  $\{g_j\}_{j\in J}$ be smooth functions on $ M $, such that for some real number 
	$ 0 < A < area(M)/2 $ we have
	\begin{enumerate}
		\item[(1)] The support of each $ f_i $ lies in some topological disc of area not greater than $ A $, and 
		similarly, the support of each $ g_j $ lies in some topological disc of area not greater than $ A $. 
		\item[(2)] $\sum_{i\in I} |f_i|\geq 1$ and $\sum_{j\in J} |g_j|\geq 1$.  
	\end{enumerate}
	Then, (\ref{eq:gen_cover}) holds for $\{f_i\}$, $\{g_j\}$.
	
	To see this, notice first that the proof of Theorem~\ref{thm:gen_cover} holds for non-negative functions with the above properties. Therefore, given arbitrary functions $\{f_i\}$, $\{g_j\}$ that satisfy these conditions, one can construct non-negative functions in the following way. Fix $\delta>0$ sufficiently small and let $\rho:\R\rightarrow [0,\infty)$ be a smooth even function satisfying
	\begin{itemize}
		\item $\rho(t) = 0$ for $t\in [-\delta, \delta]$,
		\item $\rho(t)\geq |t|-2\delta$ for all $t\in \R$,
		\item $\rho'(t)\leq 1$ for all $t\in\R$. 
	\end{itemize}   
	Setting $\tilde f_i:= (1-2|I|\delta)^{-1}\cdot \rho\circ f_i$ and $\tilde g_j:= (1-2|J|\delta)^{-1}\cdot \rho\circ g_j$, they are clearly non-negative and they are supported in $U_i$ and $V_j$ respectively. In addition, for any $x\in M$,
	\begin{eqnarray*}
		\sum_{i\in I} \tilde f_i (x) &=& (1-2|I|\delta)^{-1} \sum_{i\in I} \rho\circ f_i\\
		&\geq&  (1-2|I|\delta)^{-1} \sum_{i\in I} (|f_i| - 2\delta) \\
		&\geq&  (1-2|I|\delta)^{-1} (1-2|I|\delta) = 1.
	\end{eqnarray*}
	Similarly, $\sum_j \tilde g_j(x)\geq 1$ and hence we may apply Theorem~\ref{thm:gen_cover} to the functions $\{\tilde f_i\}_i$, $\{\tilde g_j\}_j$ and conclude 
	\begin{eqnarray*}
		\frac{ area(M)}{2A} &\leq& \sum_{i,j} \int_M  |\{\tilde f_i, \tilde g_j\}|\ \omega\\
		&=& (1-2|I|\delta)(1-2|J|\delta)\sum_{i,j} \int_M  |\{\rho\circ f_i, \rho\circ g_j\}|\ \omega \\
		&\leq& (1-2|I|\delta)(1-2|J|\delta)\sum_{i,j} \int_M  |\{f_i, g_j\}|\ \omega.
	\end{eqnarray*} 
	Taking $\delta\rightarrow 0$ we obtain (\ref{eq:gen_cover}).	
\end{rem}

Let us explain how to deduce Corollary~\ref{cor:c0_bound} from Theorem~\ref{thm:gen_cover}.
\begin{proof}[Proof of Corollary~\ref{cor:c0_bound}]
	Applying Theorem~\ref{thm:gen_cover}' for $\cU = \cV$ and $\{f_i\} = \{g_j\}$ we obtain
	\begin{equation*}
	\int_M\sum_{i,j\in I} |\{f_i, f_j\}|\ \omega \geq \frac{area(M)}{2e(\cU)}
	\end{equation*}
	The Poisson bracket of two functions $\{f_i,f_j\}$ is supported in the intersection of their supports $supp(f_i)\cap\supp(f_j)\subset U_i\cap U_j$. Hence, given $x\in M$, the function $\{f_i,f_j\}$ does not vanish at $x$ only if $x\in U_i\cap U_j$. Therefore, by the definition of $d$, the number of non-vanishing terms in the sum $\sum_{i,j} |\{f_i,f_j\}|(x)$ is at most $d^2$. We conclude that  
	\begin{equation*}\label{eq:bound_degree}
	\int_M \sum_{i,j\in I}|\{f_i, f_j\}|\ \omega \leq  area(M) \max_M \sum_{i,j\in I} |\{f_i, f_j\}| 
	\leq d^2 area(M) \max_{i,j\in I} \|\{f_i, f_{j}\}\|.
	\end{equation*}
\end{proof}

The following example shows that the bounds appearing in Theorem~\ref{thm:essential} and Corollary \ref{cor:pb_bound_gen} are sharp.

\begin{example}\label{exa:essential_sharp} 
	Let $(M,\omega)$ be a closed and connected symplectic surface. In order to demonstrate the sharpness of the bounds presented in Theorem~\ref{thm:essential} and 
	Corollary \ref{cor:pb_bound_gen}, let us construct a family of open covers $\{\mathcal U^\epsilon\}_{\epsilon>0}$ of $M$ by topological discs, such that 
	$pb(\mathcal U^\epsilon)\leq C/\epsilon^2$, $|I_{ess}(\mathcal U^\epsilon)|\geq c/\epsilon^2$ and $ c \epsilon^2 \leq area(U_i)\leq C \epsilon^2 $, for some constants $0< c < C<\infty$. 
	
	During our construction below, $ c $ and $ C $ will denote two constants, that can in principle vary from time to time. At the end of the construction we set the value 
	of $ c $ to be the minimal, and the value of $ C $ to be the maximal, among all possible values of $ c $ and of $ C $ that appeared in the construction, respectively. 
	For the sake of convenience, the euclidean norm on $ \R^2 $ is denoted by $ | \cdot | $. Throughout the construction, we use the convenient notations for a pushforward 
	and a pullback of a map\footnote{Assume that we are given a map $ \phi : A \rightarrow B $. Then for any subset $ C \subset B $ and a map 
	$ f : C \rightarrow D $, the {\it pullback} of $ f $ by $ \phi $ is the map $ \phi^* f = f \circ \phi : \phi^{-1}(C) \rightarrow D $. If in addition, $ \phi $ is injective, then for 
	any $ C \subset A $ and a map $ g : C \rightarrow D $, the {\it pushforward} of $ g $ by $ \phi $ is the map $ \phi_* f = f \circ \phi^{-1} : \phi(C) \rightarrow D $.}.
	
	Consider a cover of $ M $ by Darboux charts $ \phi_i : W_i \rightarrow U_i \subset M $, $ i=1,\ldots, m $, where $ W_i \subset \mathbb{R}^2 $ is an open set. 
	For each $ i $, choose $ W_i' \Subset W_i $, such that $ \{ W_i' \}  $ is still a cover of $ M $, and moreover for some topological disc $ D \subset W_1' $ we 
	have $ \phi_1(D) \cap \phi_i(W_i') = \emptyset $ for all $ i \neq 1 $. 
	
	Let $ h : \R^2\rightarrow\R$ be a smooth non-negative bump function compactly supported in $ \R^2 $ such that:
	\begin{itemize}
		\item The integer translations of $\{ h > 0 \} = h^{-1}((0,\infty)) $ cover $\R^2$, namely, setting $ h_\tau(z):= h(z -\tau)$, $ \tau \in \Z^2 $, we 
		have $\cup_{\tau \in \Z^2}\ \{ h_\tau > 0 \} = \R^2$. 
		\item There exists a topological disc $ V \supset supp(h) $, such that the integer translations $ \{ V+ \tau \}_{\tau \in \Z^2} $ form a minimal cover of $ \R^2 $.
	\end{itemize} 
	Further, denote $ h_{\epsilon,\tau} (z) = h_\tau(z/ \epsilon) = h(z/\epsilon-\tau) $ and $ V_{\epsilon,\tau} = \epsilon \tau + \epsilon V \subset \R^2 $ for every $ \epsilon > 0 $ and $ \tau \in \Z^2 $. Note that $ | \nabla h_{\epsilon,\tau} | \leqslant C/\epsilon $ on $ \mathbb{R}^2 $.  
	
	Let $ \epsilon > 0 $ be small enough. Consider the functions $ g = (\phi_i)_* h_{\epsilon,\tau} $, for those $ \tau \in \Z^2 $ when the support of $ h_{\epsilon,\tau} $ is contained in $ W_i' $, where by extending by $ 0 $ we view each such $ g $ as a function $ g : M \rightarrow \R $. For every such $ g $, denote $ U_g := \phi_i(V_{\epsilon,\tau}) $. Collect all such functions $ g $ for all $ 1 \leqslant i \leqslant m $, and denote them by $ g_1, \ldots , g_N $, also setting $ U_i := U_{g_i} $. By smoothness of $ \phi_i $ and since $ W_i' \Subset W_i $ for each $ i $, we have $ | \nabla \phi_i^*g_j | \leqslant C/\epsilon $ each time when the support of $ g_j $ intersects $ \phi_i(W_i') $. 
	
	Denote $ G := \sum_{i=1}^N g_i $. For small $ \epsilon $ we have $ 0 < c \leqslant G \leqslant C $. Moreover, since for each $ p \in M $ all but at most $ C $ functions among $ g_1, \ldots, g_N $ vanish on a neighbourhood of $ p $, we conclude that 
	for any choice of $ x_1, \ldots, x_N \in [-1,1] $ we have $ | \nabla \phi_i^* (x_1 g_1 + \ldots + x_N g_N) | \leqslant C / \epsilon $ on $ W_i' $. In particular, $ | \nabla \phi_i^* G | \leqslant C / \epsilon $ on $ W_i' $.  Denoting $ f_i = g_i / G $, we conclude that 
	for any choice of $ x_1, \ldots, x_N \in [-1,1] $ we have 
	\begin{gather*}
	| \nabla \phi_i^* (x_1 f_1 + \ldots + x_N f_N) |  \\
	\leqslant | \nabla \phi_i^* (x_1 g_1 + \ldots + x_N g_N) |/ \phi_i^*G + |x_1 g_1 + \ldots + x_N g_N| \cdot |\nabla \phi_i^*G |/(\phi_i^*G)^2 \\ 
	\leqslant C / \epsilon 
	\end{gather*}
	on $ W_i' $.
	
	We claim that $ \mathcal U := \{ U_i \}_{i=1, \ldots, N} $ is a desired cover of $ M $, when $ \epsilon > 0 $ is small enough. First, it is a cover of $ M $ since $ \{ W_i' \} $ is a cover of $ M $.
	Second, clearly $ c \epsilon^2 \leq area(U_i)\leq C \epsilon^2 $ for all $ i $. Since for the topological disc $ D \subset W_1' $ we have $ \phi_1(D) \cap \phi_i(W_i') = \emptyset $ for all $ i \neq 1 $, and since $ \{ V_{\epsilon,\tau} \}_{\tau \in \Z^2} $ is a minimal cover of $ \R^2 $, it follows that $|I_{ess}(\mathcal U^\epsilon)|\geq c/\epsilon^2$. Finally, to show that $pb(\mathcal U)\leq C/\epsilon^2$, we use the subordinate to $ \mathcal U $ 
	partition of unity $ f_1, \ldots, f_N $ that we constructed. Let $ x=(x_1,\ldots,x_N), y = (y_1, \ldots,y_N) \in [-1,1]^N $. Given any $ p \in M $, choose $ i $ such that $ p \in W_i' $, and put $ z = \phi_i^{-1}(p) $. We have
	\begin{gather*}
	| \{ x_1 f_1 + \ldots + x_N f_N, y_1 f_1 + \ldots + y_N f_N \}(p) | \\
	= | \{ \phi_i^* (x_1 f_1 + \ldots + x_N f_N) , \phi_i^* (y_1 f_1 + \ldots + y_N f_N) \}(z) | \\
	\leqslant | \nabla \phi_i^*(x_1 f_1 + \ldots + x_N f_N)| \cdot | \nabla \phi_i^* (y_1 f_1 + \ldots + y_N f_N) | \leqslant C/\epsilon^2. 
	\end{gather*}
	This implies that $ \| \{ x_1 f_1 + \ldots + x_N f_N, y_1 f_1 + \ldots + y_N f_N \} \| \leqslant C/ \epsilon^2 $ for all $ x,y \in [-1,1]^N $, so we get $pb(\mathcal U)\leq C/\epsilon^2$.
\end{example}

\section{Poisson brackets of small covers.}
As Leonid Polterovich explained to us, Proposition~\ref{pro:pb_c0} is a surprising application of the $C^0$-rigidity of Poisson bracket.

\begin{proof}[Proof of Proposition~\ref{pro:pb_c0}]
	For any large number $R>0$, fix functions $g,h:M\rightarrow[0,1]$ such that $\|\{g,h\}\|\geq R$. Let $\ell>0$ be larger than the Lipschitz constants of both $g$ and $h$ (with respect to the metric $\rho$). 
	Let $\cU^\epsilon:=\{U_i\}_{i=1}^N$ be a finite open cover such that the diameter of each $U_i$ is less than $\epsilon$. For each $i\in\{1,\dots,N\}$, pick $z_i \in U_i$, and notice that for every point $z\in U_i$, $|g(z)-g(z_i)|< \ell\epsilon$ and $|h(z)-h(z_i)|<\ell\epsilon$. Let $\mathcal F=\{f_i\}_{i=1}^N$ be any partition of unity subordinate to $\mathcal U^\epsilon$ and put $x:=(g(z_1),\dots,g(z_N))\in[0,1]^N$, $y:=(h(z_1),\dots,h(z_N))\in[0,1]^N$. Then, the functions $\sum_i x_i f_i$, $\sum_j y_j f_j$ are $\epsilon$-close to $g,h$ respectively: For any point $z\in M$, $|g(z) - \sum_i x_i f_i(z)| = |\sum_{i=1}^N (g(z)-g(z_i))f_i(z)|$. Since $f_i(z)=0$ whenever $z\notin U_i$ and $|g(z)-g(z_i)|< \ell\epsilon$ whenever $z\in U_i$, we have $|g(z) - \sum_i x_i f_i(z)|\leq \sum_i |g(z)-g(z_i)|f_i(z)\leq \sum_{i=1}^N \ell\epsilon f_i(z) = \ell\epsilon$. Similarly, $|h(z) - \sum_j y_j f_j(z)|\leq\ell\epsilon$ and hence the functions  $\sum_i x_i f_i$, $\sum_j y_j f_j$ converge uniformly to $g,h$ respectively when $\epsilon\rightarrow0$. The $C^0$-rigidity of Poisson brackets guarantees that for small enough $\epsilon$, $\|\{\sum_i x_i f_i(z), \sum_j y_j f_j(z)\}\|\geq R/2$. This argument holds for any subordinate partition of unity $\mathcal F$, and therefore, $pb(\mathcal U^\epsilon)\geq R/2$.
\end{proof}\

\begin{proof}[Proof of Theorem~\ref{thm:rate}]

        Let us extend a bit the notion of the Poisson bracket invariant of a cover. Given a symplectic manifold $ (M,\omega) $, a compact subset $ K \subset M $, and a finite open cover $ \cU = \{ U_i \}_{i \in I} $ of $ K $,
we define 
\begin{equation*}\label{eq:pbrel_def}
pb(\mathcal{U};K):= \inf_{\mathcal F} \max_{x,y\in[-1,1]^{|I|}} \Big\|\Big\{\sum_{i\in I} x_i f_i,\sum_{j \in I} y_j f_j\Big\}\Big\|,
\end{equation*}
where the infimum is taken over all collections $ \mathcal F = \{ f_i \}_{i \in I} $ of smooth non-negative functions satisfying $ \sum_{i \in I} f_i = 1 $ on $ K $.

	Now consider $\R^{2n}$ with the standard euclidean metric $\rho_0$, the standard symplectic form $\omega_0$ and the standard symplectic coordinates $q_1,p_1,\dots, q_n, p_n$. Let $B\subset\R^{2n}$ be the closed unit ball and let $\beta:B\rightarrow [0,1]$ be a radial bump function $\beta = \beta(r)$ that vanishes in a neighborhood of the boundary $\partial B$ and equals 1 on a neighborhood of the origin. Consider the functions $g:= q_1\cdot \beta$, $h:= p_1\cdot \beta$, then the uniform norm of their Poisson bracket (as functions on $B$) is at least 1,  namely,
	$\|\{g,h\}\| \geq 1$.
	As before, $C^0$-rigidity of Poisson brackets guarantees the existence of a constant $\delta_0>0$ (depending on $g$ and $h$, and therefore on the dimension $2n$) such that for any cover $\mathcal U^{\delta_0}$ of $B$ consisting of open sets of diameter at most $\delta_0$, $pb(\mathcal U^{\delta_0};B)\geq\frac{1}{2} \|\{g,h\}\|\geq\frac{1}{2}$. Without loss of generality we can assume that $ \delta_0 < 1/2 $.
	
	Given any symplectic manifold  $(M,\omega)$  with a compatible Riemannian metric $\rho$, there exists $0<r\leq 1$ and a symplectic embedding $\varphi:rB\rightarrow M$. Clearly, $\varphi^*(\omega) = \omega_0$ and $\varphi^*(\rho)$ is compatible with $\omega_0$. One can check that there exists a linear symplectomorphism $T:\R^{2n}\rightarrow\R^{2n}$ such that $T^*(\varphi^*\rho) (0)=\rho_0$. 
	Denote $\hat\rho:= T^*\varphi^*\rho$, then by decreasing $r$, one can guarantee that 
$ \phi \circ T $ is well defined on $ rB $ and that $\hat{\rho}|_{rB}\geq \frac{1}{2}\rho_0|_{rB}$. 

Let $ 0 < \epsilon < r/4 $. Given an open cover $\mathcal U^\epsilon$ of $M$ consisting of sets of diameter at most $ \epsilon $ with respect to the metric $ \rho $, we denote by $\mathcal V$ the collection of pre-images $ (\phi \circ T)^{-1}(U) $ for all $ U \in \cU $ with $ U \subset \phi \circ T (rB) $. Note that $ \mathcal V $ in particular contains the pre-images by $ \phi \circ T $ of all $ U \in \cU $ intersecting $ \phi \circ T(\frac{r}{2}B) $, hence $pb(\mathcal V; \frac{r}{2}B) \leq pb(\mathcal U^\epsilon)$. In addition, the diameter of every $V \in \mathcal V$ with respect to $\rho_0$ is at most $2\epsilon$. 

Denote by $\psi_c:\R^{2n}\rightarrow\R^{2n}$ the homothetic transformation $\psi_c(x) = cx$. When $\epsilon\leq \delta_0\cdot r/4=: \delta (M,\rho)$, $ \cV' := \{ \psi_{\delta_0/2\epsilon}(V) \, | \, V \in \cV \} $ is an open cover of $ B $, where the diameter of each $ V' \in \cV' $ is at most $\delta_0$ with respect to $\rho_0$. Therefore, by our previous arguments $pb(\cV';B)\geq 1/2$. In addition, for any $c>0$ and any cover $\mathcal U$ of a compact set $ K \subset \mathbb{R}^{2n} $ we have $$ pb(\{ \psi_c(U) \, | \, U \in \cU \}; \psi_c(K)) = 1/c^2 \cdot pb(\mathcal U;K).$$ In particular, 
	\begin{equation*}
	\frac{1}{2}\leq pb(\cV';B) = \frac{4\epsilon^2}{\delta_0^2} pb(\mathcal V; \frac{2\epsilon}{\delta_0} B)\leq 
	\frac{4\epsilon^2}{\delta_0^2} pb(\mathcal V; \frac{r}{2} B) \leq
	\frac{4\epsilon^2}{\delta_0^2} pb(\mathcal U^\epsilon).
	\end{equation*}
	The required bound easily follows.
	
\end{proof}

\appendix
\section{Bounding $pb(\mathcal U)$ by the sum of absolute values of Poisson brackets.} \label{app:lin_exc}

\begin{prop}\label{pro:linear_ineq}
	Let $(\R^{2n},\omega_0)$ be the standard symplectic vector space, and let $v_1,\dots v_N\in \R^{2n}$ be vectors that satisfy $\sum_{i=1}^N \| v_i\|\leq \sum_{i=1}^N \|Sv_i\|$ for all $S\in \operatorname{Sp}(n)$. Then, there exists a constant $c(n)>0$ depending only on the dimension such that 
	\begin{equation}\label{eq:linear1}
	\max_{x,y\in [-1,1]^N}\omega_0\left(\sum_{i=1}^N x_i v_i, \sum_{j=1}^N y_j v_j \right)\geq c(n)\cdot \left(\sum_{i=1}^N \|v_i\|\right)^2 .
	\end{equation}
\end{prop}\

\begin{cor}\label{cor:linear}
	Let $v_1,\dots ,v_N\in \R^{2n}$ be any collection of vectors, then 
	\begin{equation}\label{eq:linear2}
	\max_{x,y\in [-1,1]^N}\omega_0\left(\sum_{i=1}^N x_i v_i, \sum_{j=1}^N y_j v_j \right)\geq c(n)\cdot \sum_{i,j=1}^N |\omega_0(v_i, v_j)|. 
	\end{equation}
\end{cor}

\noindent We will first prove Corollary \ref{cor:linear} using Proposition \ref{pro:linear_ineq}.

\begin{proof}[Proof of Corollary \ref{cor:linear}]
	First, let us show that we may assume that the vectors $v_1,\dots,v_N$ span $\R^{2n}$. Denote $V:=span\{v_1,\dots,v_N\}$, then it can be decomposed into a symplectically orthogonal direct sum of a symplectic vector space and an isotropic one: $V=V_S\oplus V_I$. Denoting by $P:V\rightarrow V_S$ the projection, both sides of (\ref{eq:linear2}) are invariant under the replacement $v_i\mapsto P v_i$ (since $\omega_0|_{V_I}=0$, and since $ V_I $ and $ V_S $ are symplectically orthogonal). Clearly the vectors $\{P v_i\}$ span the symplectic vector space $V_S$. Replacing $\{v_i\}$ with $\{Pv_i\}$ and $(\R^{2n},\omega_0)$ with $(V_S,\omega_0)$, and recalling that  $(V_S,\omega_0)\cong (\R^{2n'},\omega_0)$ for some $n'\leq n$, justifies the assumption that the vectors $\{v_i\}_{i=1}^N$ span $\R^{2n}$.
	
	Consider the map $\operatorname{Sp}(n)\rightarrow\R$, $S\mapsto \sum_{i=1}^N \|S  v_i\|$. This map is continuous with respect to the operator norm, and we claim that it admits a minimum on $\operatorname{Sp}(n)$, denoted $S_{min}$. To see this, take any $S\in\operatorname{Sp}(n)$ with $\|S\|_{op}\geq L$, then the maximal eigenvalue of $S^TS$ satisfies $\lambda\geq L^2$. Let $u\in\R^{2n}$ be a corresponding unit eigenvector of $S^TS$, namely $S^T Su=\lambda u$. Then for every $v\in\R^{2n}$, 
	\begin{equation*}
	\|Sv\|^2=\left<Sv,Sv\right> = \left<v,S^TSv\right> \geq \lambda\cdot\left<u,v\right>^2\geq L^2\cdot\left<u,v\right>^2.
	\end{equation*} 
	Therefore, $\sum_{i=1}^N\|S v_i\|\geq L\cdot \sum_i |\left<u,v_i\right>|\geq L\cdot a$, where  
	\begin{equation*}
	a:=\min_{\|w\|=1} \sum_{1\leq i\leq N} |\left<w,v_i\right>| >0
	\end{equation*}
	is independent of $S$ (note that $a>0$ due to our assumption, that $\{v_i\}_i$ span $\R^{2n}$). This yields a lower bound for the map $S\mapsto \sum_{i=1}^N \|S  v_i\|$ which grows with operator norm. Therefore, the map  $S\mapsto \sum_{i=1}^N \|S  v_i\|$ indeed admits a minimum. 
	The minimizing matrix $S_{min}$ is symplectic, and therefore both sides of (\ref{eq:linear2}) are invariant under composition with it. By replacing $\{v_i\}$ with $\{S_{min}v_i\}$ we may assume that $\sum_{i=1}^N \| v_i\|\leq \sum_{i=1}^N \|Sv_i\|$ for every $S\in \operatorname{Sp}(n)$. Applying Proposition~\ref{pro:linear_ineq} and recalling that $\sum_{i,j=1}^N |\omega_0(v_i, v_j)|\leq \sum_{i,j=1}^N \|v_i\|\cdot\| v_j\| = (\sum_{i=1}^N \|v_i\|)^2$ yields (\ref{eq:linear2}).
\end{proof}\

\begin{rem} \label{rem:linear}
Let $ V $ be a real vector space of dimension $ 2n $, endowed with a non-degenerate $ 2 $-form $ \omega $.
Since $ (V,\omega) $ is linearly isomorphic to $ (\mathbb{R}^{2n},\omega_0) $, Corollary \ref{cor:linear} implies that for any $ v_1, \ldots, v_N \in V $ we have 
	\begin{equation*}
	\max_{x,y\in [-1,1]^N}\omega \left(\sum_{i=1}^N x_i v_i, \sum_{j=1}^N y_j v_j \right)\geq c(n)\cdot \sum_{i,j=1}^N |\omega(v_i, v_j)|.
	\end{equation*}
\end{rem}

\begin{proof}[Proof of Lemma~\ref{lem:pb_vs_sum}]
	Let $(M^{2n},\omega)$  be  a closed symplectic manifold and let $\{f_i\}_{i\in I}\subset C^\infty(M)$ be a finite collection of smooth functions. Without loss of generality, assume $I=\{1,\dots, N\}$ for some $N\in\N$. Let $p\in M$ be any point.


	Applying Remark~\ref{rem:linear} to $ V = T_pM $ and $ v_i = X_{f_i}(p) $ (where $ X_{f_i} $ denotes the symplectic gradient of $ f_i $) yields
	\begin{eqnarray*}
		c(n)\cdot \sum_{i,j=1}^N |\{f_i,f_j\}|(p) &=& c(n) \cdot \sum_{i,j=1}^N |\omega(X_{f_i}(p),X_{f_j}(p))| \\
		&\leq&		\max_{x,y\in[-1,1]^N} \omega\Big(\sum_i x_i X_{f_i}(p), \sum_j y_j X_{f_j}(p)\Big)\\
		&=&		\max_{x,y\in[-1,1]^N} \Big\{\sum_i x_i {f_i}, \sum_j y_j {f_j}\Big\}(p)
	\end{eqnarray*}

If we take $ p \in M $ to be the point where $ \max_M \sum_{i,j=1}^N |\{f_i,f_j\}| $ is achieved, we conclude 
	\begin{eqnarray*}
		c(n)\cdot \max_M \sum_{i,j=1}^N |\{f_i,f_j\}| &=& c(n)\cdot \sum_{i,j=1}^N |\{f_i,f_j\}|(p) \\
		&\leq& \max_{x,y\in[-1,1]^N} \Big\{\sum_i x_i {f_i}, \sum_j y_j {f_j}\Big\}(p) \\
		&\leq& \max_{x,y\in[-1,1]^N} \Big\|\Big\{\sum_i x_i {f_i}, \sum_j y_j {f_j}\Big\}\Big\|.
	\end{eqnarray*}
\end{proof}

Before we prove Proposition~\ref{pro:linear_ineq}, let us present some notations.
Let $\{C_j\}_{j=1}^{m}$ be a collection of $m=m(n,\theta)$ cones of angle $\theta \in (0,\pi /4) $ that cover the space. Namely, there exist unit vectors $\{z_j\}$ such that 
$$
C_j:=\left\{u\in\R^{2n}:\frac{\left<u,z_j\right>}{\|u\|}\geq \cos\theta\right\},
$$ 
and $\cup_{j=1}^mC_j=\R^{2n}$. Below we refer to $z_j$ as the center of $C_j$.
Let $C\in \{C_j\}_{j=1}^{m}$ be a cone with maximal sum of norms, namely for all  $1\leq j\leq m$,
\begin{equation*}
\sum_{v_i\in C_j} \|v_i\| \leq \sum_{v_i\in C}\|v_i\|.
\end{equation*}

\begin{proof}[Proof of Proposition~\ref{pro:linear_ineq}]
	First, let us notice that it is enough to prove that there exists a constant $A(n,\theta)>0$ such that
	\begin{equation}\label{eq:bnd_w_norms_in_C}
	\left(\sum_{v_i\in C} \|v_i\|\right)^2 \leq A(n,\theta)\cdot \max_{x,y\in[-1,1]^N} \omega_0\left(\sum_{i=1}^N x_iv_i, \sum_{j=1}^N y_jv_j\right).
	\end{equation} 
	Indeed, by our choice of $C$ we have 
	\begin{eqnarray*}
		\left(\sum_{i=1}^N \|v_i\|\right)^2
		&\leq& m(n,\theta)^2 \cdot \left(\sum_{v_i\in C} \|v_i\|\right)^2\\
		&\overset{(\ref{eq:bnd_w_norms_in_C})}{\leq}&  m(n,\theta)^2 \cdot A(n,\theta)\cdot \max_{x,y\in[-1,1]} \omega_0\left(\sum_{i=1}^N x_iv_i, \sum_{j=1}^N y_jv_j\right).
	\end{eqnarray*}
	Setting $c:= (m^2\cdot A)^{-1}$ yields the proposition. Therefore, it remains to prove (\ref{eq:bnd_w_norms_in_C}).
	Set $v:=\left(\sum_{v_i\in C} v_i\right)/\|\sum_{v_i\in C} v_i\|$ and consider the orthogonal decomposition 
	$\R^{2n}=span(v)\oplus span(J_0 v)\oplus\big(span(v)\oplus span(J_0 v)\big)^\perp$. Here $ J_0 : \mathbb{R}^{2n} \rightarrow \mathbb{R}^{2n} $ is the linear map 
	satisfying $ \omega(\xi,\eta) = \left< \xi,J_0 \eta \right> $ for any $ \xi, \eta \in \mathbb{R}^{2n} $, where $ \left< \cdot,\cdot \right> $ is the scalar product on 
	$ \mathbb{R}^{2n} $ (thus, $ J_0 $ is the ``multiplication by the imaginary unit'', if we naturally identify $ \mathbb{R}^{2n} \cong \mathbb{C}^n $).

	Let $S\in \operatorname{Sp}(n)$ be the symplectic linear map defined by $S(av+bJ_0v+w) = \frac{1}{2}av+ {2}b J_0v+w$ for $a,b\in\R$ and 
	$w\in (span(v)\oplus span(J_0 v))^\perp$. Then for every $ u \in \mathbb{R}^{2n} $, looking at the decomposition $u=av+bJ_0v+w\in\R^{2n}$ as above, we get
	\begin{eqnarray*}
		\|Su\|-\|u\| &=& \frac{\|Su\|^2 -\|u\|^2}{\|Su\|+\|u\|} \leq \frac{\|Su\|^2 -\|u\|^2}{\|u\|}\\
		&=& \frac{a^2/4 + 4b^2 + \|w\|^2- a^2 - b^2 - \|w\|^2}{\|u\|}\\
		&\leq& \frac{3b^2}{\|u\|} \leq 3|b|=  {3|\omega_0(u,v)|}.
	\end{eqnarray*} 
	For $u\in C$, $a=\left<u,v\right>\geq \cos(2\theta)\cdot\|u\|$ and
	\begin{eqnarray*}
		\|Su\| &=& \sqrt{a^2/4 +4b^2 +\|w\|^2}\\
		&\leq&  \sqrt{\frac{1}{4}\cos(2\theta)^2+4\sin(2\theta)^2}\cdot{\|u\|}\\
		&\leq& \frac{2}{3} \|u\|,
	\end{eqnarray*} 
	where the last inequality holds when we take $\theta$ small enough (e.g. $\theta\leq\pi/30$). 
	Writing $v_j = a_j v+b_jJ_0v +w_j$, we have 
	\begin{eqnarray*}
		0 &\leq& \sum_{j=1}^N \|Sv_j\| -\|v_j\|\\ 
		&=& \sum_{v_j\in C} \|Sv_j\| -\|v_j\| + \sum_{v_j\notin C} \|Sv_j\| -\|v_j\| \\
		&\leq& -\frac{1}{3} \sum_{v_j\in C}\|v_j\| + 3\sum_{v_j\notin C} |\omega_0(v,v_j)|.
	\end{eqnarray*}
	We conclude that $$ \sum_{v_j\in C}\|v_j\| \leq 9\sum_{v_j\notin C} |\omega_0(v,v_j)| .$$
	Recall that $ v=\left(\sum_{v_i\in C} v_i\right)/\|\sum_{v_i\in C} v_i\|$.
	Setting $x_i=1$ if $v_i\in C$ and 0 otherwise, $y_j=\text{sign}\,\omega_0(v,v_j)$ if $v_j\notin C$ and 0 otherwise, we have
	\begin{equation*}
	\Big\|\sum_{v_i\in C} v_i\Big\|\cdot \left(\sum_{v_j\in C} \|v_j\|\right)  \leq 9\ \omega_0\left(\sum_{i=1}^N x_iv_i, \sum_{j=1}^N y_j v_j\right).
	\end{equation*}
	
	Since $C$ is a cone of angle $\theta$ with a center $ z $, for any $u \in C$ we have $\left<u, z\right> \geq \cos(\theta)\cdot \|u\|$. Therefore, 
	$$\Big\|\sum_{v_i\in C} v_i\Big\| \geq \left<\sum_{v_i\in C} v_i, z\right>  = \sum_{v_i\in C} \left<v_i, z\right> \geq  \cos(\theta) \cdot \sum_{v_j\in C} \|v_j\|.$$
	Combining the above inequalities we obtain
	\begin{equation*}
	\left(\sum_{v_j\in C} \|v_j\|\right)^2  \leq \frac{9}{\sqrt{\cos(\theta)}}\cdot\max_{x,y\in[-1,1]^N}\omega_0\left(\sum_{i=1}^N x_i v_i, \sum_{j=1}^N y_j v_j\right).
	\end{equation*}
This proves (\ref{eq:bnd_w_norms_in_C}) and hence the proposition.
\end{proof}	


\section{Proof of Lemma \ref{lem:pb_to_intersections}} \label{app:lemma-intgeom-proof}

	We prove the lemma in three steps. \vspace{2mm}
	
	\noindent \underline{Step 1:} We first show the statement of the lemma for measurable subsets $\Omega \subset \R^2 \setminus cv(\Phi)$. For doing that, it is enough to consider 
	the case when $ \Omega $ is compactly contained in $ \R^2 \setminus cv(\Phi) $. Indeed, any given $\Omega \subset \R^2 \setminus cv(\Phi)$ can 
	be exhausted by a non-decreasing sequence of measurable subsets compactly contained in $ \R^2 \setminus cv(\Phi)$. In 
	particular, $(\ref{eq:pb_to_intersections})$  holds for each subset from the sequence, and by passing to the limit, we conclude that $ \Omega $ 
	satisfies $(\ref{eq:pb_to_intersections})$ as well.
	
	Moreover, by partitioning $ \Omega $ into small pieces, we may assume that $ \Omega $ is a subset of a sufficiently small neighbourhood of a 
	given point $ z \in \R^2 \setminus cv(\Phi)$. Now let $ z \in \R^2 \setminus cv(\Phi)$, then $ z $ is a regular value of the map $ \Phi $, and hence for a small 
	neighbourhood $ U $ of $ z $, the preimage $ \Phi^{-1}(U) $ is a disjoint union $ \Phi^{-1}(U) = \cup_{k=1}^N V_k $ of open subsets $ V_k \subset M $, 
	where $ \Phi|_{V_k} : V_k \rightarrow U $ is a diffeomorphism for each $ k $. Note that $ N = K(s,t) $ for any $ (s,t) \in U $. As a result, for any 
	measurable $ \Omega \subset U $, we have 
	
	\begin{eqnarray*}
		\int_{\Phi^{-1}(\Omega)} |\{f,g\}|\, \omega 
		&=& \sum_{k=1}^N \int_{V_k \cap \Phi^{-1}(\Omega)} |df \wedge dg| \\	
		&=& \sum_{k=1}^N \int_{V_k} \id_{\Phi^{-1}(\Omega)} \cdot |df \wedge dg| \\
		&=& \sum_{k=1}^N \int_U \id_{\Omega}\ ds\, dt\\
		&=& N \cdot \int_\Omega\ ds\, dt\\
		&=& \int_{\Omega} K(s,t)\ ds\, dt.\\
	\end{eqnarray*} 
	

\noindent 	\underline{Step 2:} Assume that $\Omega\subset cv(\Phi)$, then by Sard's theorem, $\Omega$ is of measure zero, and hence the right hand side of (\ref{eq:pb_to_intersections}) vanishes. Denote by $Z\subset \Phi^{-1}(cv(\Phi))$ the set of regular points of $\Phi$ in $\Phi^{-1}(cv(\Phi))$. Since $\Phi$ restricts to a diffeomorphism on a neighborhood of each point in $Z$, and $\Phi(Z)\subset cv(\Phi)$ is of measure zero in $\R^2$, we conclude that $Z$ is of measure zero in $M$ (again, with respect to the volume density given by $\omega$). In addition, $|\{f,g\}|$ equals to the Jacobian of $\Phi$ and thus vanishes on $cp(\Phi)$. Therefore, for $\Phi^{-1}(\Omega) \subset \Phi^{-1}(cv(\Phi)) = Z\cup cp(\Phi)$, we have 
	\begin{equation*}
	\int_{\Phi^{-1}(\Omega)} |\{f,g\}|\ \omega \leq \int_Z |\{f,g\}|\ \omega + \int_{cp(\Phi)} |\{f,g\}|\ \omega = 0.
	\end{equation*}  
	We conclude that the left hand side of (\ref{eq:pb_to_intersections}) vanishes as well, which proves the claim for this case.\\ \\
	\underline{Step 3:} Assume $\Omega\subset \R^2$ is measurable, and consider the decomposition $\Omega = \big(\Omega\setminus cv(\Phi)\big)\cup \big(\Omega\cap cv(\Phi)\big)$. Then, by the previous steps, (\ref{eq:pb_to_intersections}) holds for both $\Omega\setminus cv(\Phi)$ and $\Omega\cap cv(\Phi)$, and therefore holds for their disjoint union as well.

\bibliographystyle{plain}
\bibliography{refs}
\vspace*{1cm}

\paragraph{Lev Buhovski,}$ $\\
School of Mathematical Sciences\\ 
Tel Aviv University \\
Ramat Aviv, Tel Aviv 69978\\
Israel\\
E-mail: levbuh@post.tau.ac.il

\paragraph{Alexander Logunov,}$ $\\ 
School of Mathematics \\ 
Institute for Advanced Study \\ 
Princeton, NJ 08540 \\ \\ 
School of Mathematical Sciences \\ 
Tel Aviv University \\ Tel Aviv 69978 \\ Israel \\ \\ Chebyshev Laboratory \\ St. Petersburg State University \\ 14th Line V.O., 29B, Saint Petersburg 199178 \\ Russia. \\ \\ 
E-mail: log239@yandex.ru

\paragraph{Shira Tanny,}$ $\\
School of Mathematical Sciences\\ 
Tel Aviv University \\
Ramat Aviv, Tel Aviv 69978\\
Israel\\
E-mail: tanny.shira@gmail.com

\end{document}